\documentclass[11pt]{amsart}
\usepackage[T1]{fontenc}
\usepackage{lmodern,microtype}
\usepackage[a4paper,margin=29mm]{geometry}
\usepackage{amssymb,mathtools}
\usepackage[hidelinks]{hyperref}
\hypersetup{
  pdftitle={Sharp bounds and gap phenomena for large p-subgroups of automorphism groups of curves},
  pdfauthor={Saeed Tafazolian},
  pdfkeywords={Automorphism groups of curves, positive characteristic, p-rank, Artin-Schreier curves}
}
\newtheorem{theorem}{Theorem}[section]
\newtheorem{proposition}[theorem]{Proposition}
\newtheorem{lemma}[theorem]{Lemma}
\newtheorem{corollary}[theorem]{Corollary}
\theoremstyle{remark}
\newtheorem{remark}[theorem]{Remark}
\newtheorem*{remark*}{Remark}

\newcommand{\Aut}{\operatorname{Aut}}
\newcommand{\PSL}{\operatorname{PSL}}
\newcommand{\PGL}{\operatorname{PGL}}
\newcommand{\GL}{\operatorname{GL}}
\newcommand{\PSU}{\operatorname{PSU}}
\newcommand{\PGU}{\operatorname{PGU}}
\newcommand{\SU}{\operatorname{SU}}

\newcommand{\Ree}{\operatorname{Ree}}
\newcommand{\F}{\mathbb F}
\newcommand{\PP}{\mathbb P}
\newcommand{\PhiG}{\Phi}

\title[Sharp bounds for large $p$-subgroups]
{Sharp bounds and gap phenomena for large $p$-subgroups of automorphism groups of curves}

\author{Saeed Tafazolian}
\address{Departamento de Matem\'atica, Instituto de Matem\'atica, Estat\'istica e Computa\c{c}\~ao Cient\'ifica (IMECC), Universidade Estadual de Campinas (UNICAMP), Campinas, SP, Brazil}
\email{tafazolian@ime.unicamp.br}
\date{}

\begin{document}

\begin{abstract}
Let $X$ be a curve of genus $g\ge2$ over an algebraically closed field of odd characteristic $p$, and let $G=\Aut(X)$. A theorem of Giulietti and Korchm\'aros shows that if $G$ fixes no point and a $p$-subgroup $S\le G_P$ satisfies
\[
 |S|>\frac{p}{p-1}g,
\]
then $X$ is one of the classical hyperelliptic, Hermitian, or Ree curves. We determine the sharp boundary and the first gap below it. At equality, besides the classical curves, exactly three types occur: the ordinary curves $y^p-y=x+c/x$, the zero-$p$-rank curves $y^3=x^p-x$ for $p\equiv2\pmod3$, and, in characteristic $3$, the ordinary genus-six family
\[
 y^3-y=c\left(x+\frac1{x^3-x}\right).
\]
As a key ingredient, we classify zero-$p$-rank curves whose full automorphism group fixes no point and which admit a $p$-subgroup of order larger than the genus. We then prove that, when $|S|>p$, the boundary classification already follows from the weaker inequality $|S|>p(g-p+1)/(p-1)$. For $|S|\ge p^3$, the forbidden interval above $m=(p-1)|S|/p$ has width at least $p(p-1)/2$, and below $m+(p^2-1)/2$ at most one further genus can occur, with a rigid ramification structure. Finally, we classify the first admissible layer: it occurs only for $|S|=p^2$, and the resulting ordinary curves admit a complete description in terms of two-dimensional Artin--Schreier spaces.
\end{abstract}


\maketitle

\section{Introduction}

Let $k$ be an algebraically closed field of characteristic $p>0$, and let $X/k$ be a smooth projective curve of genus $g\ge2$. When $p$ divides the order of the automorphism group, the classical Hurwitz bound no longer controls $|\Aut(X)|$, and the local action of large $p$-subgroups becomes decisive. A natural question is therefore how large a $p$-subgroup of a point stabilizer can be before the curve is forced into one of the classical highly symmetric families.

A theorem of Giulietti and Korchm\'aros \cite[Theorem~1.1]{GK10} gives a sharp starting point for this question. It is stated there for the first ramification group $G_P^{(1)}$; since $G_P^{(1)}$ is the unique Sylow $p$-subgroup of $G_P$, it is equivalent to the formulation below for a $p$-subgroup $S\le G_P$. Restricting to odd characteristic, assume that
\[
 G=\Aut_k(X)
\]
has no global fixed point and that
\begin{equation}\label{eq:GK-strict}
 |S|>\frac{p}{p-1}g.
\end{equation}
Then $X$ is one of three classical families: the hyperelliptic curve $y^2=x^n-x$, the Hermitian curve, or the Ree curve. The aim of this paper is to determine the sharp boundary of this theorem and to describe the first possible genera immediately below it. We first classify the equality case
\[
 |S|=\frac{p}{p-1}g,
\]
and then prove gap results showing that, for larger $p$-subgroups, the same classification persists under strictly weaker numerical hypotheses.

\begin{theorem}\label{thm:boundary}
Let $k$ be algebraically closed of odd characteristic $p$, let $X/k$ have genus $g\ge2$, and put $G=\Aut_k(X)$. Assume that $G$ has no global fixed point on $X$ and that, for some $P\in X$, the stabilizer $G_P$ contains a $p$-subgroup $S$ such that
\begin{equation}\label{eq:closedbound}
 |S|\ge\frac{p}{p-1}g.
\end{equation}
Then one of the following holds.
\begin{enumerate}
\item $X$ is one of the classical odd-characteristic curves in \cite[Theorem~1.1]{GK10}: the hyperelliptic curve $y^2=x^n-x$ with $n=p^a\ge5$, the Hermitian curve, or the Ree curve.
\item Equality holds in \eqref{eq:closedbound}, $g=p-1$, and
\[
 X:\quad y^p-y=x+\frac{c}{x},\qquad c\in k^*.
\]
This curve is ordinary.
\item Equality holds, $p\equiv2\pmod3$, $g=p-1$, and
\[
 X:\quad y^3=x^p-x.
\]
This curve has $p$-rank zero.
\item $p=3$, equality holds, $g=6$, $|S|=9$, and
\[
 X:\quad y^3-y=c\left(x+\frac1{x^3-x}\right),\qquad c\in k^*.
\]
This curve is ordinary.
\end{enumerate}
Conversely, each curve in \emph{(2)--(4)} realizes equality in \eqref{eq:closedbound} for a $p$-subgroup fixing a point, while its full automorphism group has no global fixed point.
\end{theorem}

Thus the strict bound of \cite{GK10} can be replaced by a non-strict one once the equality cases are added to the list. The ordinary genus-$p-1$ family in Theorem~\ref{thm:boundary}(2) is the classical Kloosterman-type Artin--Schreier family studied by van der Geer--van der Vlugt \cite{vdGvdV}.

\begin{remark*}[Relation with \cite{GK17}]
Giulietti and Korchm\'aros \cite[Theorem~1.1]{GK17} proved that if a $p$-subgroup $S$ satisfies $|S|>\frac{p^2}{p^2-p-1}(g-1)$, then one of three things happens:
\begin{itemize}
\item $X$ has $p$-rank zero and $X\to X/S$ is totally ramified at a unique place and unramified elsewhere;
\item $|S|=p$ and $X$ is ordinary of genus $p-1$;
\item $X$ is ordinary and Nakajima-extremal, and there is a subgroup $M<S$ of index $p$ such that $X\to X/M$ is unramified and $X/M$ is an ordinary curve of genus $p-1$.
\end{itemize}
Their Remarks~2.2--2.3 discuss the curves $y^p-y=ax+1/x$, their automorphism group, and the exceptional curve in characteristic $3$ with automorphism group $\GL(2,3)$.

On the boundary $g=\frac{p-1}{p}q$, where $q=|S|$, one has
\[
 q>\frac{p^2(g-1)}{p^2-p-1}\iff q<p^2 .
\]
Hence for $q=p$, the zero-$p$-rank or ordinary dichotomy in Proposition~\ref{prop:rank-dichotomy} follows from \cite[Theorem~1.1]{GK17}, since all its positive-$p$-rank alternatives give an ordinary curve of genus $p-1$. In particular, Nakajima-extremality gives $p=\frac{p}{p-2}(\gamma-1)$ and hence $\gamma=p-1=g$. In the ordinary case, the two fixed points and the rational quotient follow from Deuring--Shafarevich and Hurwitz, and reduced Artin--Schreier theory gives the Kloosterman family of Proposition~\ref{prop:ordinary-p}. Proposition~\ref{prop:kloosterman-aut} gives a self-contained treatment of the automorphism groups discussed in \cite[Remarks~2.2--2.3]{GK17} and adds the isomorphism criterion.

For $q\ge p^2$ and $g\ge m=(p-1)q/p$, one has
\[
 \frac{p^2(g-1)}{p^2-p-1}
 \ge \frac{p(p-1)q-p^2}{p^2-p-1}
 \ge q,
\]
so the strict numerical hypothesis of \cite[Theorem~1.1]{GK17} fails. Thus the genus-six boundary family and the regions $g>m$ studied in Theorems~\ref{thm:gap}--\ref{thm:firstlayer} lie outside its range.
\end{remark*}

The point here is the complete description of the boundary under the hypothesis that the full automorphism group has no global fixed point. The proof separates naturally according to the $p$-rank. At the boundary, the Deuring--Shafarevich formula forces a dichotomy: either the $p$-rank is zero or the curve is ordinary. The zero-$p$-rank part is supplied by the following stronger result.

The zero-$p$-rank branch belongs naturally to the Harbater--Katz--Gabber setting. We use the extra-automorphism theorem of Bleher--Chinburg--Poonen--Symonds \cite[Theorems~1.8 and~1.11]{BCPS}. Their Example~6.2 already contains the curves $z^n=t^q-t$ with $n\mid(q+1)$, including our cubic family.

Their theorem treats solvable groups. For nonsolvable groups, \cite[Remark~7.15]{BCPS} refers to the structural theorem of Guralnick--Malmskog--Pries \cite[Theorem~3.16]{GMP}. Its geometric consequence, \cite[Theorem~3.19]{GMP}, applies in our odd-characteristic, zero-$p$-rank setting when a Sylow $p$-subgroup $Q$ is noncyclic and $N_G(Q)\ne G$: it gives a normal $p'$-subgroup $M=Z(N_G(Q))$ such that $G/M$ is almost simple with socle in an explicit list. The underlying group-theoretic proof uses the classification of finite simple groups.

Our argument instead follows the orbit method of \cite{GK10}. It treats the cyclic-Sylow cases as well, and it determines the genus and an explicit model.

The same type of cover appears in the theory of big actions. Lehr--Matignon \cite{LM05} and Matignon--Rocher \cite{MR08} study pairs $(X,S)$ with $S$ a $p$-group and $|S|/g>2p/(p-1)$; then $S$ fixes a point and $X\to X/S\cong\PP^1$ is ramified only there. Theorem~\ref{thm:zerorank} concerns such covers in the weaker range $|S|>g$, under the global hypothesis that $\Aut(X)$ fixes no point. For zero $p$-rank as a consequence of large automorphism groups, see also \cite{GK19}; for the characteristic-$2$ zero-rank case, see \cite{GKjlms}.

The cubic family is also a classical Hermitian quotient when $3\mid(q+1)$; see, for example, \cite{GSX}. None of these curves is claimed to be new. The new statements are:
\begin{itemize}
\item the boundary classification for $|S|\ge p^2$ (for $|S|=p$ see the Remark above);
\item Theorem~\ref{thm:zerorank} and the automorphism groups of the cubic and genus-six families;
\item the isomorphism criterion for the Kloosterman family;
\item the numerical gaps below the boundary and the structure theorem beyond the refined gap;
\item the first-layer classification.
\end{itemize}

Ordinary curves with large automorphism groups have been studied by Korchm\'aros--Montanucci \cite{KM19} and Montanucci--Speziali \cite{MS}. The Artin--Schreier curves $y^q-y=x+c/x$, which contain our first-layer family, appear as extremal examples in \cite{KM19}. Our first-layer theorem starts instead from the numerical condition above and the hypothesis that the full automorphism group has no global fixed point, forces $|S|=p^2$, and classifies the corresponding rank-two Artin--Schreier spaces.

\begin{theorem}\label{thm:zerorank}
Let $k$ be algebraically closed of odd characteristic $p$, and let $X/k$ have genus $g\ge2$ and $p$-rank zero. Put $G=\Aut_k(X)$, and let $S\le G_P$ be a $p$-subgroup for some $P\in X$. Assume that $|S|>g$ and that $G$ has no global fixed point on $X$. Then $X$ is $k$-isomorphic to the smooth projective model of one of the following curves:
\begin{enumerate}
\item $y^2=x^n-x$, where $n=p^a\ge5$;
\item the Hermitian curve $y^n+y=x^{n+1}$, where $n=p^a\ge3$;
\item the Ree curve, where $p=3$, $n=3n_0^2$, and $n_0=3^a$;
\item $y^3=x^q-x$, where $q=p^a\equiv2\pmod3$.
\end{enumerate}
Conversely, all these curves satisfy the hypotheses. In case \emph{(4)}, $g=q-1$, $|S|=q$, and
\[
 \Aut_k(X)\cong C_3\times\PGL(2,q).
\]
\end{theorem}

\begin{corollary}\label{cor:zerogap}
Under the hypotheses of Theorem~\ref{thm:zerorank}, if $X$ is not one of the three classical families, then $|S|=g+1$ and $|S|\equiv2\pmod3$. In particular, $|S|\ge g+2$ forces one of the classical families.
\end{corollary}

There is also a genuine numerical gap immediately below the boundary of Theorem~\ref{thm:boundary}.

\begin{theorem}[Gap theorem]\label{thm:gap}
Let $k$ be algebraically closed of odd characteristic $p$, let $X/k$ have genus $g\ge2$, and put $G=\Aut_k(X)$. Assume that $G$ has no global fixed point and that $G_P$ contains a $p$-subgroup $S$ of order $q>p$. If
\begin{equation}\label{eq:gap-bound}
 q>\frac{p}{p-1}\bigl(g-(p-1)\bigr),
\end{equation}
then in fact
\[
 q\ge\frac{p}{p-1}g.
\]
Consequently $X$ is one of the classical curves in \cite[Theorem~1.1]{GK10}, or
\[
 p=3,\qquad q=9,\qquad g=6,
\]
and
\[
 X:\quad y^3-y=c\left(x+\frac1{x^3-x}\right),\qquad c\in k^*.
\]
Equivalently, if $m=(p-1)q/p$, then no such curve exists with
\[
 m<g<m+(p-1).
\]
\end{theorem}

\begin{theorem}[Refined gap for higher $p$-power order]\label{thm:refinedgap}
Let $k$ be algebraically closed of odd characteristic $p$, let $X/k$ have genus $g\ge2$, and put $G=\Aut_k(X)$. Assume that $G$ has no global fixed point on $X$ and that, for some $P\in X$, the stabilizer $G_P$ contains a $p$-subgroup $S$ of order $q\ge p^3$. Put
\[
 m=\frac{p-1}{p}q.
\]
Then no such curve exists with
\[
 m<g<m+\frac{p(p-1)}2.
\]
Equivalently, if
\[
 q>\frac{p}{p-1}
 \left(g-\frac{p(p-1)}2\right),
\]
then
\[
 q\ge\frac{p}{p-1}g,
\]
and hence the boundary classification of Theorem~\ref{thm:boundary} applies.
\end{theorem}

\begin{theorem}[Beyond the refined gap]\label{thm:beyond}
Let $k$ be algebraically closed of odd characteristic $p$, let $X/k$ have genus $g\ge2$, and put $G=\Aut_k(X)$. Assume that $G$ has no global fixed point on $X$ and that, for some $P\in X$, the stabilizer $G_P$ contains a $p$-subgroup $S$ of order $q\ge p^3$. Put $m=(p-1)q/p$. If
\[
 m<g<m+\frac{p^2-1}{2},
\]
then $g=m+\frac{p(p-1)}2$, and the following hold.
\begin{enumerate}
\item $X$ has $p$-rank $m$, and $X/S\cong\PP^1$.
\item Besides $P$, the group $S$ has exactly one short orbit $\Delta$. It has length $q/p$, and for $R\in\Delta$ the stabilizer $L=S_R$ has order $p$ and trivial second ramification group.
\item $S_P^{(2)}=\cdots=S_P^{(p+1)}=Z$ has order $p$, and $S_P^{(p+2)}=1$.
\item $Z\neq L$, and the cover $X\to X/L$ is weakly ramified with $g(X/L)-\gamma(X/L)=\frac{p-1}{2}$.
\end{enumerate}
\end{theorem}

For $p=3$ one has $\frac{p^2-1}2=\frac{p(p-1)}2+1$, so Theorem~\ref{thm:beyond} adds no forbidden value but describes the first possible value. For $p\ge5$ it forbids, in addition, every $g$ with
\[
 m+\tfrac{p(p-1)}2<g<m+\tfrac{p^2-1}2.
\]
Whether the configuration of Theorem~\ref{thm:beyond} occurs is open; see Remark~\ref{rem:endpoint}.

\begin{theorem}[Classification of the first admissible layer]\label{thm:firstlayer}
Let $k$ be algebraically closed of odd characteristic $p$, let $X/k$ have genus $g\ge2$, and put $G=\Aut_k(X)$. Assume that $G$ has no global fixed point on $X$ and that, for some $P\in X$, the stabilizer $G_P$ contains a $p$-subgroup $S$ of order $q>p$. Put
\[
 m=\frac{p-1}{p}q.
\]
If
\[
 g=m+(p-1),
\]
then necessarily
\[
 q=p^2,\qquad g=\gamma(X)=p^2-1.
\]
Moreover $S\cong C_p^2$, the quotient $X/S$ is rational, and $S$ has exactly two fixed points; the cover $X\to X/S$ is totally and weakly ramified at both of them and unramified elsewhere.

Choose a coordinate $x$ on $X/S$ sending the two branch points to $0$ and $\infty$, and write
\[
 \wp(z)=z^p-z.
\]
Then
\[
 k(X)=k(x)\bigl(\wp^{-1}(V)\bigr),
\]
where $V$ is a two-dimensional $\F_p$-subspace of
\[
 \left\{ax+\frac b x:\ a,b\in k\right\}
\]
such that every nonzero element has both coefficients nonzero. The normalizer $N_G(S)$ contains an element interchanging the two fixed points of $S$. After multiplying $x$ by a scalar, the induced involution on $X/S$ is $x\mapsto1/x$, and $V$ is invariant under
\[
 J\!\left(ax+\frac b x\right)=bx+\frac a x.
\]
Conversely, every such $J$-invariant subspace $V$ defines an ordinary curve of genus $p^2-1$ satisfying the hypotheses.

Since $p$ is odd, after a further projective change preserving the unordered branch pair, $V$ has one of the following two normal forms:
\begin{enumerate}
\item
\[
 V=A\left(x+\frac1x\right),
\]
where $A\subset k$ is a two-dimensional $\F_p$-subspace;
\item
\[
 V=\F_p\,a\left(x+\frac1x\right)
 \oplus
 \F_p\,b\left(x-\frac1x\right),
\]
where $a,b\in k^*$ and $a/b\notin\F_p$.
\end{enumerate}
In the first case the involution $x\mapsto1/x$ acts trivially on $V$; in the second it has eigenvalues $1$ and $-1$. The special family
\[
 y^{p^2}-y=x+\frac{c}{x}
\]
is contained in case \emph{(1)}. Thus the strict inequality in \eqref{eq:gap-bound} is optimal as a bound uniform in $q$, and equality at the first layer occurs only for $q=p^2$.
\end{theorem}

\begin{corollary}[Isomorphism classes and normalizers in the first layer]\label{cor:firstlayer-moduli}
Let $X$ and $X'$ be first-layer curves as in Theorem~\ref{thm:firstlayer}, choose Sylow $p$-subgroups $S\le\Aut(X)$ and $S'\le\Aut(X')$, and let $V,V'\subset k^2$ be the corresponding Artin--Schreier planes after sending the two branch points to $0,\infty$. Then $X\cong X'$ if and only if, for some $\lambda\in k^*$,
\[
 V'=D_\lambda(V)
 \quad\text{or}\quad
 V'=J D_\lambda(V),
\]
where
\[
 D_\lambda(a,b)=(\lambda^{-1}a,\lambda b),
 \qquad
 J(a,b)=(b,a).
\]
Thus the (unmarked) first-layer isomorphism classes are the orbits of the admissible two-dimensional $\F_p$-subspaces of $k^2$ under $k^*\rtimes C_2$. Moreover, for a first-layer curve $X=X(V)$,
\[
 N_{\Aut(X)}(S)/S
 \cong
 \operatorname{Stab}_{\,k^*\rtimes C_2}(V),
\]
where $k^*\rtimes C_2$ is the group of projectivities preserving the unordered branch pair $\{0,\infty\}$.
\end{corollary}

Corollary~\ref{cor:firstlayer-moduli} implies that each of the two normal forms of Theorem~\ref{thm:firstlayer} contributes a two-parameter family (Remark~\ref{rem:moduli}).

Sections~2--6 prove Theorem~\ref{thm:zerorank}. Sections~7 and~8 treat the equality case without any assumption on the $p$-rank and prove Theorem~\ref{thm:boundary}. Section~9 proves Theorems~\ref{thm:gap}--\ref{thm:firstlayer} and Corollary~\ref{cor:firstlayer-moduli}. The positive-$p$-rank part of these arguments is purely local; it is isolated in Proposition~\ref{prop:local}.

\section{Preliminaries}

Throughout the paper $p>2$. For a finite subgroup $H\le\Aut(X)$ and a point $P\in X$, we write $H_P$ for the stabilizer of $P$ and $H_P^{(i)}$ for the $i$th lower ramification group. We use the Hurwitz and Deuring--Shafarevich formulas in their standard forms; see \cite{HKT,Serre}.

If $P$ is fixed by a group $H$, then
\[
 H_P=H_P^{(1)}\rtimes C,
\]
where $H_P^{(1)}$ is the unique Sylow $p$-subgroup of $H_P$ and $C$ is cyclic of order prime to $p$. In particular, every point stabilizer is solvable, and every subgroup of a point stabilizer whose order is prime to $p$ is cyclic.

We shall use the Kantor--O'Nan--Seitz classification of doubly transitive permutation groups with cyclic two-point stabilizer \cite{KOS}, in the form recalled in \cite[Section~2]{GK10} and made precise in Lemma~\ref{lem:nofield}. We also use the classification of finite $p$-irregular subgroups of $\PGL(2,k)$; a convenient modern reference is \cite{Faber}.

\begin{lemma}\label{lem:onepoint}
Let $X$ have genus $g\ge0$ and $p$-rank zero. All $p$-subgroups below are finite.
\begin{enumerate}
\item Every nontrivial $p$-subgroup of $\Aut(X)$ has a unique fixed point and acts semiregularly away from it.
\item Let $S$ be a $p$-subgroup of order $q>g$, and let $P$ be its fixed point. Then
\[
 X/S\cong\PP^1.
\]
If $S_i=S_P^{(i)}$, then
\begin{equation}\label{eq:genusS}
 2g=\sum_{i\ge2}(|S_i|-1).
\end{equation}
\item Assume $g\ge2$ and retain the notation and hypothesis $q=|S|>g$ of \emph{(2)}. If $G=\Aut(X)$, $I=G_P$, and $S$ is a Sylow $p$-subgroup of $G$, then the different exponent at $P$ for $X\to X/G$ is
\begin{equation}\label{eq:different}
 d_P=2g-2+q+|I|.
\end{equation}
\end{enumerate}
\end{lemma}

\begin{proof}
Let $C\le\Aut(X)$ have order $p$, and let $r$ be the number of fixed points of $C$. Deuring--Shafarevich gives
\[
 -1=p(\gamma'-1)+(p-1)r,
\]
where $\gamma'$ is the $p$-rank of $X/C$. Hence $r=1$ and $\gamma'=0$.

Now let $U$ be a nontrivial $p$-subgroup. Choose a central element $z\in U$ of order $p$, and let $P$ be its unique fixed point. Since $z$ is central, $u(P)$ is fixed by $z$ for every $u\in U$, and therefore $u(P)=P$. Thus $P$ is fixed by all of $U$. If $1\ne u\in U$ fixed another point $Q$, then a suitable power of $u$ of order $p$ would also fix $Q$, contradicting the first paragraph. This proves (1).

Let $g_0$ be the genus of $X/S$. The cover $X\to X/S$ is ramified only at $P$. Since $S_0=S_1=S$, its different exponent satisfies
\[
 d_P(S)\ge2(q-1).
\]
If $g_0\ge1$, Hurwitz gives
\[
 2g-2\ge2(q-1),
\]
so $g\ge q$, a contradiction. Hence $g_0=0$. Applying Hurwitz again gives
\[
 2g-2=-2q+2(q-1)+\sum_{i\ge2}(|S_i|-1),
\]
which is \eqref{eq:genusS}. None of these steps uses $g\ge2$.

Finally, $I_0=I$, $I_1=S$, and $I_i=S_i$ for $i\ge1$. Therefore
\[
 d_P=(|I|-1)+(q-1)+\sum_{i\ge2}(|S_i|-1),
\]
and \eqref{eq:different} follows from \eqref{eq:genusS}.
\end{proof}

We shall also use the following consequence of the theory of Harbater--Katz--Gabber curves.

\begin{lemma}\label{lem:nonsolvable}
Under the hypotheses of Theorem~\ref{thm:zerorank}, the full group $G$ is not solvable.
\end{lemma}

\begin{proof}
Let $P$ be the unique fixed point of $S$. By Lemma~\ref{lem:onepoint}, the pair $(X,P)$ is an HKG $S$-curve: the quotient $X/S$ is rational and the cover is ramified only at $P$.

Suppose that $G$ is solvable. Since $G$ moves $P$, the pair $(X,P)$ is not an HKG $G$-curve by \cite[Theorem~1.8]{BCPS}, and the solvable extra-automorphism theorem \cite[Theorem~1.11]{BCPS} applies. Because $p$ is odd and $g\ge2$, the only remaining curve in their list is $z^4=t^3u-tu^3$ in characteristic $3$.

Dehomogenizing gives $z^4=t^3-t$. Choose $\iota\in\F_9$ with $\iota^2=-1$ and put $t=\iota x$. Then $t^3-t=-\iota(x^3+x)$, and after rescaling $z$ by a fourth root of $-\iota$ we obtain the genus-three Hermitian curve $y^4=x^3+x$. Its full automorphism group is $\PGU(3,3)$, hence is not solvable. This contradiction proves the assertion.
\end{proof}

For later use we isolate the rational-line consequence of Dickson's classification.

\begin{lemma}\label{lem:dickson}
Let $H\le\PGL(2,k)$ be a finite nonsolvable group whose order is divisible by $p$. Assume that a Sylow $p$-subgroup $U$ of $H$ has order $q=p^a$, and that $H$ has a doubly transitive orbit $B\subset\PP^1(k)$ of size $q+1$ on which $U$ fixes one point and acts regularly on the remaining $q$ points. Then, after a projective change of coordinate,
\[
 H\cong\PSL(2,q)\quad\text{or}\quad\PGL(2,q),
\qquad
 B=\PP^1(\F_q).
\]
\end{lemma}

\begin{proof}
By the classification of finite $p$-irregular subgroups of $\PGL(2,k)$ \cite{Faber}, a nonsolvable such group is isomorphic to $\PSL(2,p^e)$ or $\PGL(2,p^e)$, apart from the exceptional group $A_5$ in characteristic $3$. Comparison of Sylow $p$-orders gives $p^e=q$ in the linear cases. In the exceptional case $q=3$, but $A_5$ has no subgroup of index $4$, so it has no orbit of size $q+1=4$.

Thus $H$ is a standard linear group over $\F_q$, up to conjugacy in $\PGL(2,k)$. Let $b\in B$. Since $H_b$ has index $q+1$, its order is divisible by the full $p$-part $q$, and hence $H_b$ contains a Sylow $p$-subgroup $U_b$. A nontrivial $p$-subgroup of $\PGL(2,k)$ has a unique fixed point on $\PP^1(k)$; therefore that fixed point is $b$. It follows that $B$ is exactly the set of fixed points of the Sylow $p$-subgroups of $H$. For the standard groups $\PSL(2,q)$ and $\PGL(2,q)$ this set is $\PP^1(\F_q)$. A projective change of coordinate therefore identifies $B$ with $\PP^1(\F_q)$.
\end{proof}

\section{The orbit of the \texorpdfstring{$p$}{p}-fixed points}

The orbit reduction in this section follows the framework of \cite[Section~3]{GK10}: one separates the case in which the $p$-fixed orbit is the unique short orbit from the case with one further tame short orbit, and then uses the action on $X/S$ to obtain double transitivity. The difference here is that the strict numerical hypothesis of \cite{GK10} is replaced by the zero-$p$-rank Lemma~\ref{lem:onepoint}; we include the details because this is precisely what allows the weaker bound $|S|>g$.

In Sections~3--6 we prove Theorem~\ref{thm:zerorank}. Starting with the $p$-subgroup in its statement, choose a Sylow $p$-subgroup $S^*$ of $G$ containing it. By Lemma~\ref{lem:onepoint}, $S^*$ has a unique fixed point; this point is fixed by the original subgroup and hence is $P$. Thus $S^*\le G_P$ and $|S^*|>g$. We may therefore replace the original subgroup by $S^*$ and, throughout these sections, denote it again by $S$.

Let $q=|S|$, let $P$ be the unique fixed point of $S$, and let
\[
 \Omega=\{R\in X: R\text{ is fixed by an element of order }p\}.
\]
By Lemma~\ref{lem:onepoint}(1), each Sylow $p$-subgroup has a unique fixed point. Conversely, if a point is fixed by an element of order $p$, it is the fixed point of a Sylow $p$-subgroup containing that element. Two distinct Sylow $p$-subgroups cannot have the same fixed point, since a point stabilizer has a unique Sylow $p$-subgroup. Hence $\Omega$ is naturally identified with the set of Sylow $p$-subgroups of $G$. In particular,
\[
 N_G(S)=G_P=:I,
\]
and $\Omega$ is a $G$-orbit.

The group $S$ fixes $P$ and acts freely on $\Omega\setminus\{P\}$. Thus
\begin{equation}\label{eq:omega-bq}
 |\Omega|=1+bq
\end{equation}
for some integer $b\ge1$. Write
\[
 I=S\rtimes H,
\qquad h=|H|.
\]

\begin{proposition}\label{prop:orbits}
The action of $G$ on $\Omega$ is doubly transitive, and
\[
 |\Omega|=q+1.
\]
Apart from $\Omega$, the group $G$ has at most one short orbit, and such an orbit is tame. If $Q$ belongs to that tame short orbit, put
\[
 w=|G_Q|;
\]
if no such orbit exists, put $w=1$. Then
\begin{equation}\label{eq:master}
 2g-2=(q+1)\left(\frac{h}{w}-1\right),
\end{equation}
and
\begin{equation}\label{eq:master2}
 2g=\frac{(q+1)h}{w}-(q-1).
\end{equation}
Moreover, in the tame-orbit case, $G_Q$ acts semiregularly on $\Omega$, and therefore
\[
 w\mid q+1.
\]
\end{proposition}

\begin{proof}
Since $X/S$ is rational, L\"uroth's theorem gives $X/G\cong\PP^1$. Every short orbit different from $\Omega$ is tame, because a point whose stabilizer contains a nontrivial $p$-element belongs to $\Omega$.

Suppose that there are two tame short orbits in addition to $\Omega$. Each contributes at least $|G|/2$ to the Hurwitz different. If $r=|\Omega|$, then $|G|=r|I|$, and Lemma~\ref{lem:onepoint} gives
\[
\begin{aligned}
 2g-2
 &\ge -2|G|+r d_P+|G|\\
 &=r(d_P-|I|)\\
 &=r(2g-2+q),
\end{aligned}
\]
a contradiction. Hence there is at most one short orbit besides $\Omega$.

Assume first that $\Omega$ is the unique short orbit. Hurwitz gives
\[
 2g-2=-2|G|+r d_P
 =r(2g-2+q-|I|).
\]
Thus $r$ divides $2g-2$. On the other hand, \eqref{eq:omega-bq} and $q>g$ give
\[
 r\ge q+1>g-1.
\]
Every proper divisor of $2g-2$ is at most $g-1$, so
\[
 r=2g-2.
\]
Now \eqref{eq:omega-bq} and $q>g$ force $b=1$. Hence $r=q+1$, and $S$ is regular on $\Omega\setminus\{P\}$. Since $G$ is transitive on $\Omega$, the action is doubly transitive. Equation \eqref{eq:master} follows from Hurwitz with $w=1$. In this case $q=2g-3$ and $|I|=2q$, that is, $h=2$; this is the configuration $\PSL(2,5)$, $g=4$, $w=1$ treated in Lemma~\ref{lem:cubic}.

Assume now that there is one further short orbit $\Delta$, and choose $Q\in\Delta$. Put
\[
 a=|I|,
\qquad
 w=|G_Q|.
\]
Hurwitz gives
\[
 2g-2
 =r\left(2g-2+q-\frac{a}{w}\right).
\]
Set
\[
 N=w(2g-2+q)-a.
\]
Then $N>0$ and
\begin{equation}\label{eq:rN}
 r=\frac{(2g-2)w}{N}.
\end{equation}

We first claim that no point of $\Omega\setminus\{P\}$ has trivial stabilizer in $I$. Otherwise, for some $R\in\Omega\setminus\{P\}$ the $I$-orbit of $R$ would have size $a$, so $r\ge a+1$. Equation \eqref{eq:rN} would give
\[
 N<\frac{(2g-2)w}{a}.
\]
Also $N\ge1$ and $r\ge a+1$ give $a<(2g-2)w$, and hence
\[
 N=w(2g-2+q)-a>wq.
\]
Consequently
\[
 aq<2g-2.
\]
But $a\ge q$ and $q>g$, so $aq\ge q^2>2g-2$, a contradiction.

Thus every $R\in\Omega\setminus\{P\}$ is fixed by a nontrivial tame element of $I$. Fix such an $R$ and choose an element $\alpha\in I_R$ of prime order. Since $I_R$ is a tame two-point stabilizer, it is cyclic. By Schur--Zassenhaus, $\alpha$ is contained in a cyclic complement $H_R$ of $S$ in $I$.

The element $\alpha$ induces a nontrivial automorphism of $X/S\cong\PP^1$. It fixes the images of $P$ and of the $S$-orbit $SR$, and these are therefore its two fixed points. Since $H_R$ is cyclic, it centralizes $\alpha$; it also fixes the image of $P$, and therefore preserves the other fixed point of $\alpha$. Hence $H_R$ preserves $SR$, and so does
\[
 I=SH_R.
\]
Thus every $S$-orbit contained in $\Omega\setminus\{P\}$ is $I$-invariant.

Let there be $\nu$ such orbits. Since each of their point stabilizers in $I$ is nontrivial and tame, $h>1$. Choose $1\ne\beta\in H$. The induced nontrivial automorphism of $X/S$ fixes the image of $P$ and the images of all these $\nu$ orbits. A nontrivial projectivity has at most two fixed points, so $\nu=1$. Hence
\[
 |\Omega|=q+1,
\]
and the action of $G$ on $\Omega$ is doubly transitive.

It remains to prove the assertion about $G_Q$. Suppose that a nonidentity $\alpha\in G_Q$ fixes a point of $\Omega$. After conjugating in $G$, we may assume that it fixes $P$. Then $\alpha\in I$. Since $G_Q$ is tame, $\alpha$ has order prime to $p$. On $X/S$, the element $\alpha$ fixes the images of $P$, of the unique $S$-orbit $\Omega\setminus\{P\}$, and of $SQ$. These three points are distinct. Therefore $\alpha$ induces the identity on $X/S$, and hence $\alpha\in S$, a contradiction. Thus $G_Q$ is semiregular on $\Omega$, and $w\mid q+1$.

Finally, substituting $|\Omega|=q+1$, $|I|=qh$, and \eqref{eq:different} into Hurwitz gives \eqref{eq:master}; equation \eqref{eq:master2} is equivalent to it.
\end{proof}

\begin{proposition}[Subgroup form of the orbit reduction]\label{prop:orbits-subgroup}
Let $X$ have $p$-rank zero and genus $g\ge2$, let $\Gamma\le\Aut(X)$ be a finite subgroup with no global fixed point on $X$, and let $U$ be a Sylow $p$-subgroup of $\Gamma$ of order $q>g$. Let $P$ be the unique fixed point of $U$, and let
\[
 \Omega_\Gamma
 =
 \{R\in X:R\text{ is fixed by an element of order }p\text{ in }\Gamma\}.
\]
Then the conclusions of Proposition~\ref{prop:orbits} hold with
\[
 (G,S,\Omega,I,h)
 \quad\text{replaced by}\quad
 (\Gamma,U,\Omega_\Gamma,\Gamma_P,h_\Gamma).
\]
In particular,
\[
 |\Omega_\Gamma|=q+1,
\]
the action of $\Gamma$ on $\Omega_\Gamma$ is doubly transitive, there is at most one further short orbit and it is tame, and if
\[
 \Gamma_P=U\rtimes C_{h_\Gamma}
\]
and $w_\Gamma$ denotes the stabilizer order in the tame short orbit (or $w_\Gamma=1$ if there is none), then
\[
 2g-2
 =
 (q+1)\left(\frac{h_\Gamma}{w_\Gamma}-1\right).
\]
\end{proposition}

\begin{proof}
The proof of Proposition~\ref{prop:orbits} uses only the action of the group under consideration, not the maximality of the full automorphism group. Indeed, Lemma~\ref{lem:onepoint} gives the unique fixed point of $U$ and the rationality of $X/U$; since $U$ is Sylow in $\Gamma$, it is the first ramification group of $\Gamma_P$. The different formula, the short-orbit estimate, the fixed-point argument on $X/U\cong\PP^1$, and the semiregularity argument then go through verbatim with $\Gamma$ in place of $G$. The hypothesis that $\Gamma$ fixes no point guarantees $\Omega_\Gamma\ne\{P\}$, i.e.\ $b\ge1$.
\end{proof}

Let
\[
 K=\ker(G\to\operatorname{Sym}(\Omega)),
\qquad t=|K|.
\]
Since $K$ fixes two points of $\Omega$, it lies in a two-point stabilizer. Hence $K$ is cyclic and $p\nmid t$. By Proposition~\ref{prop:orbits}, no nonidentity element of $K$ fixes a point outside $\Omega$. Therefore the tame cover
\[
 X\longrightarrow Y:=X/K
\]
is totally ramified at the $q+1$ points of $\Omega$ and unramified elsewhere. If $g_1=g(Y)$, Hurwitz gives
\begin{equation}\label{eq:kernel}
 2g=2tg_1+(t-1)(q-1).
\end{equation}
Since $g<q$,
\[
 (t-1)(q-1)\le2g\le2(q-1),
\]
so
\begin{equation}\label{eq:tbound}
 t\le3.
\end{equation}
If $t=3$, equality holds and
\begin{equation}\label{eq:t3genus}
 g=q-1,
\qquad
 g_1=0.
\end{equation}

The faithful permutation group
\[
 \overline G=G/K
\]
has a cyclic two-point stabilizer. By Kantor--O'Nan--Seitz, and since $p$ is odd, either $\overline G$ has a regular normal subgroup, or $L\le\overline G\le\Aut(L)$ with $L$ one of $\PSL(2,q)$, $\PSU(3,n)$ ($q=n^3$), or $\Ree(n)$ ($p=3$, $q=n^3$) in its natural action. By Lemma~\ref{lem:nofield} below, the nonsolvable possibilities are therefore
\[
 \PSL(2,q),\quad \PGL(2,q),\quad
 \PSU(3,n),\quad \PGU(3,n)\ (q=n^3),\quad
 \Ree(n)\ (p=3,\ q=n^3).
\]
The alternative with a regular normal subgroup, which is then elementary abelian, is solvable; because $K$ is cyclic, it would make $G$ solvable, contrary to Lemma~\ref{lem:nonsolvable}.

\begin{lemma}\label{lem:nofield}
Let $L$ be one of the groups
\begin{enumerate}
\item $\PSL(2,q)$ on $\PP^1(\F_q)$, with $q=p^a$;
\item $\PSU(3,n)$ on the $n^3+1$ isotropic points of $\PP^2(\F_{n^2})$, with $n=p^a$;
\item $\Ree(n)$ on its $n^3+1$ points, with $p=3$ and $n=3^{2a+1}\ge27$.
\end{enumerate}
Let $L\le\overline G\le\Aut(L)$, acting on the same set. If the two-point stabilizers of $\overline G$ are cyclic, then, respectively,
\[
 \overline G\le\PGL(2,q),\qquad \overline G\le\PGU(3,n),\qquad \overline G=\Ree(n).
\]
In particular $\overline G$ is one of $\PSL(2,q)$, $\PGL(2,q)$, $\PSU(3,n)$, $\PGU(3,n)$, $\Ree(n)$. For $n=3$, Kantor--O'Nan--Seitz gives $\overline G=\Ree(3)$ directly.
\end{lemma}

\begin{proof}
Let $P\ne R$ be two points, and let $\overline G_0$ be the intersection of $\overline G$ with $\PGL(2,q)$, $\PGU(3,n)$, or $\Ree(n)$ respectively. Since $L$ is transitive on ordered pairs of distinct points, $\overline G=L\,\overline G_{P,R}$. Hence, if $\overline G\ne\overline G_0$, then $\overline G_{P,R}$ contains an element $\beta$ inducing a nontrivial field automorphism. We show that $\beta$ does not centralize the two-point stabilizer $C_0:=L_{P,R}\le\overline G_{P,R}$. It follows that $\overline G_{P,R}$ is nonabelian, a contradiction.

\emph{Linear case.} Take $P=\infty$ and $R=0$. Then $C_0=\{x\mapsto\mu^2x:\mu\in\F_q^*\}$, and $\beta(x)=cx^{p^j}$ with $1\le j\le a-1$. One computes
\[
 \beta\circ(x\mapsto\mu^2x)\circ\beta^{-1}=(x\mapsto\mu^{2p^j}x).
\]
If $\beta$ centralized $C_0$, then $\mu^{2(p^j-1)}=1$ for all $\mu\in\F_q^*$, so $q-1\mid2(p^j-1)$. This is impossible, because $2(p^{a-1}-1)<p^a-1$.

\emph{Unitary case.} Use the Hermitian form $x_1\bar y_3+x_2\bar y_2+x_3\bar y_1$, where $\bar y=y^n$, and take $P=\langle e_1\rangle$, $R=\langle e_3\rangle$. The elements of $\SU(3,n)$ fixing $P$ and $R$ also fix $\langle e_2\rangle=\langle e_1,e_3\rangle^\perp$. Hence they are exactly
\[
 h(\lambda)=\operatorname{diag}(\lambda,\lambda^{n-1},\lambda^{-n}),
 \qquad \lambda\in\F_{n^2}^*,
\]
and $C_0$ is the image of $\{h(\lambda)\}$ in $\PSU(3,n)$. The element $h(\lambda)$ is scalar exactly when $\lambda^3=1=\lambda^{n+1}$; these $\lambda$ form a group of order $\mu=\gcd(3,n+1)$.

Write $\beta=\delta\varphi_j$, where $\delta$ is diagonal and $\varphi_j$ raises matrix entries to the power $p^j$, with $1\le j\le 2a-1$. Then
\[
 \beta h(\lambda)\beta^{-1}=h(\lambda^{p^j}).
\]
If $\beta$ centralized $C_0$, then $\lambda^{\mu(p^j-1)}=1$ for all $\lambda$, so $n^2-1\mid\mu(p^j-1)$. If $\mu=1$ this contradicts $p^j-1<n^2-1$. If $\mu=3$, then $p\ne3$, so $p\ge5$, and $3(p^{2a-1}-1)<n^2-1$, again a contradiction.

\emph{Ree case.} Here $\operatorname{Out}(\Ree(n))$ is cyclic of order $2a+1$, generated by the field automorphism $\varphi$. The two-point stabilizer $C_0=L_{P,R}$ is cyclic of order $n-1$. In the standard $7$-dimensional representation it is a diagonal torus $\{h(\lambda):\lambda\in\F_n^*\}$ whose entries are monomials in $\lambda$, and $\varphi$ acts by $h(\lambda)\mapsto h(\lambda^3)$; see \cite{Kl88}.

Choose $P,R$ to be the two standard points fixed by $\varphi$, and write $\beta=c\varphi^j$ with $1\le j\le2a$. Then $c=\beta\varphi^{-j}\in L_{P,R}=C_0$ commutes with $C_0$. If $\beta$ centralized $C_0$, then $\lambda^{3^j}=\lambda$ for all $\lambda\in\F_n^*$. This is impossible, since $3^j-1<n-1$.
\end{proof}

\section{The cases with nontrivial kernel}

\begin{proposition}\label{prop:kernel}
If $K\ne1$, then $X$ is one of the following curves:
\[
 y^2=x^q-x,
\]
or
\[
 y^3=x^q-x,
\qquad q\equiv2\pmod3.
\]
\end{proposition}

\begin{proof}
Assume first that $t=2$. Then \eqref{eq:kernel} gives
\begin{equation}\label{eq:t2}
 g=2g_1+\frac{q-1}{2}.
\end{equation}
The curve $Y$ has $p$-rank zero, since the pullback on the etale $p$-torsion of the Jacobian is injective for a cover of degree prime to $p$.

We show that $g_1=0$. Suppose, to the contrary, that $g_1>0$. Suppose first that
\[
 \overline G\cong\PSL(2,q)
 \quad\text{or}\quad
 \PGL(2,q).
\]
The complement in a point stabilizer acts irreducibly on the elementary abelian Sylow subgroup $S\cong(\F_q,+)$. Indeed, the field generated over $\F_p$ by the scalar complement is $\F_q$. Hence every nontrivial lower ramification group for the $S$-action on $Y$ is equal to $S$. If the last positive lower group is $S_j$, then $j\ge2$ because $g_1>0$, and Lemma~\ref{lem:onepoint} applied to $Y$ (which is allowed for any genus) gives
\[
 2g_1=(q-1)(j-1).
\]
Together with \eqref{eq:t2}, this yields
\[
 g\ge\frac{3}{2}(q-1)\ge q,
\]
a contradiction.

Suppose next that $\overline G$ is unitary. Let $L\cong\PSU(3,n)$ be the natural simple subgroup and let $\widetilde L$ be its inverse image in $G$. Since $K\cong C_2$, the extension is central. By the Schur-multiplier fact used in \cite[Lemma~2.7]{GK10}, it splits:
\[
 \widetilde L\cong L\times C_2.
\]
A point stabilizer in $L$ has even order, so an involution of $L$ fixing a point of $\Omega$, together with $K$, gives a subgroup $C_2\times C_2$ in a point stabilizer of $G$. This is impossible because every subgroup of order prime to $p$ in a point stabilizer is cyclic. The same argument applies to $\Ree(n)$ for $n>3$.

It remains to consider $\overline G\cong\Ree(3)$. Here $q=27$, and \eqref{eq:t2} with $g<27$ gives $g_1\le6$. If $g_1=1$, a point stabilizer of an elliptic curve in characteristic $3$ cannot contain a $3$-subgroup of order $27$. Assume $g_1\ge2$. The group $\Ree(3)\cong\PGL(2,8)\rtimes C_3$ is nonsolvable, so it cannot fix a point of $Y$, since a point stabilizer on a curve of genus at least two is solvable. Proposition~\ref{prop:orbits-subgroup} therefore applies to its action on $Y$. The tame complement in its natural point stabilizer has order $2$, and hence
\[
 2g_1-2=28(2/w-1).
\]
For $w=1$ this gives $g_1=15$, while $w=2$ gives $g_1=1$; larger $w$ gives negative genus. This contradicts $2\le g_1\le6$.

Thus $Y\cong\PP^1$. The group $\overline G\le\PGL(2,k)$ is nonsolvable and acts doubly transitively on the $q+1$ branch points. Lemma~\ref{lem:dickson} identifies this branch set with $\PP^1(\F_q)$. A double cover of $\PP^1$ branched exactly at these $q+1$ points is unique up to isomorphism, since $\operatorname{Pic}(\PP^1)$ has no $2$-torsion. Hence
\[
 X\cong\{y^2=x^q-x\}.
\]

Now suppose that $t=3$. Then $p\ne3$, and \eqref{eq:t3genus} gives
\[
 Y\cong\PP^1,
\qquad g=q-1.
\]
Again $\overline G$ is nonsolvable and has a doubly transitive orbit of size $q+1$. Lemma~\ref{lem:dickson} gives, after a projective change of coordinate,
\[
 \overline G\cong\PSL(2,q)\quad\text{or}\quad\PGL(2,q),
\qquad
 \Omega/K=\PP^1(\F_q).
\]

The conjugation action of $G$ on $K\cong C_3$ factors through $\overline G$. The subgroup $\PSL(2,q)$ is perfect for $q\ge5$, so it acts trivially on
\[
 \Aut(C_3)\cong C_2.
\]
Thus the inverse image of $\PSL(2,q)$ centralizes $K$.

Write the cyclic cover as
\[
 y^3=f(x).
\]
The local Kummer exponents at the $q+1$ branch points are all equal modulo $3$, because $\PSL(2,q)$ is transitive on the branch set and centralizes the deck group. Their sum is zero modulo $3$, so
\[
 3\mid q+1.
\]
After replacing $y$ by $y^{-1}$ at the level of Kummer characters if necessary, all exponents may be taken equal to $1$. Since $\PP^1$ has no nontrivial $3$-torsion in its Picard group, the cover is the unique cyclic cubic cover with this branch divisor. In coordinates for which the branch set is $\PP^1(\F_q)$, this gives
\[
 X\cong\{y^3=x^q-x\}.
\]
In particular, $q\equiv2\pmod3$.
\end{proof}

\section{The faithful action}

Assume from now on that $K=1$. If
\[
 q>\frac{p}{p-1}g,
\]
then \cite[Theorem~1.1]{GK10} already gives one of the classical families. Hence, to find any additional curve, it is enough to work in the strip
\begin{equation}\label{eq:strip}
 \frac{p-1}{p}q\le g<q.
\end{equation}
We show that the faithful action is impossible in this strip.

\subsection{The linear case}

Suppose
\[
 G\cong\PSL(2,q)
 \quad\text{or}\quad
 \PGL(2,q)
\]
in its natural action on $\Omega$. The Sylow subgroup is elementary abelian, and the cyclic complement of order
\[
 h=\frac{q-1}{2}
 \quad\text{or}\quad
 h=q-1
\]
acts irreducibly on it. Hence every nontrivial lower ramification group at $P$ is equal to $S$. If $j$ is the last lower jump, then
\begin{equation}\label{eq:linear-genus}
 2g=(q-1)(j-1).
\end{equation}
For an order-$p$ quotient the unique jump is prime to $p$, so $p\nmid j$. Since $g<q$, one has $j=2$ or $j=3$.

If $j=2$, then
\[
 g=\frac{q-1}{2}<\frac{p-1}{p}q,
\]
contrary to \eqref{eq:strip}. Thus
\begin{equation}\label{eq:linear-new}
 j=3,\qquad g=q-1,\qquad p\ne3.
\end{equation}
Equation \eqref{eq:master} gives
\[
 w=\frac{q+1}{6}
 \quad\text{for }\PSL(2,q),
 \qquad
 w=\frac{q+1}{3}
 \quad\text{for }\PGL(2,q),
\]
and therefore
\begin{equation}\label{eq:qmod3}
 q\equiv2\pmod3.
\end{equation}

\begin{lemma}\label{lem:cubic}
In the situation \eqref{eq:linear-new},
\[
 X\cong\{y^3=x^q-x\}.
\]
Moreover, for $q=5$ the curve $y^3=x^5-x$ has a unique trigonal pencil.
\end{lemma}

\begin{proof}
Write $q=p^a$. By \eqref{eq:qmod3}, $p\equiv2\pmod3$ and $a$ is odd. Let
\[
 h=\frac{q-1}{2}
 \quad\text{or}\quad
 h=q-1
\]
according as $G$ is $\PSL(2,q)$ or $\PGL(2,q)$. Choose a coordinate $u$ on $X/S$ such that the unique branch point is infinity and a generator of the complement acts by
\[
 u\longmapsto\zeta u,
\]
where $\zeta$ has order $h$.

We first note that
\begin{equation}\label{eq:fieldgenerated}
 \F_p(\zeta^3)=\F_q.
\end{equation}
Indeed, $\gcd(3,h)=1$, so $\zeta^3$ also has order $h$. If $\zeta^3$ belonged to a proper subfield $\F_{p^d}$, then $h\le p^d-1$. Since $a$ is odd, every proper divisor $d$ of $a$ satisfies $d\le a/3$, and
\[
 p^d-1\le p^{a/3}-1<\frac{p^a-1}{2}\le h,
\]
a contradiction.

The elementary abelian extension $k(X)/k(u)$ is represented by an $a$-dimensional $\F_p$-space $V$ of reduced Artin--Schreier polynomials. Since the unique lower jump is $3$, every nonzero element of $V$ has degree exactly $3$. Hence the leading-coefficient map on $V$ is injective.

Let $W$ be its image. The complement preserves $V$, so multiplication by $\zeta^3$ preserves $W$. By \eqref{eq:fieldgenerated}, $W$ is an $\F_q$-vector space. Since
\[
 \dim_{\F_p}W=a,
\]
we have $W=c\F_q$ for some $c\ne0$. After scaling $u$, we may assume $W=\F_q$. Thus
\begin{equation}\label{eq:ASspace}
 V=
 \{bu^3+L_2(b)u^2+L_1(b)u:b\in\F_q\},
\end{equation}
where $L_1,L_2$ are $\F_p$-linear maps. Write
\[
 L_j(b)=\sum_{i=0}^{a-1}c_{j,i}b^{p^i},
 \qquad j=1,2.
\]
Invariance under $u\mapsto\zeta u$ gives
\begin{equation}\label{eq:covariance}
 c_{j,i}\ne0
 \quad\Longrightarrow\quad
 3p^i\equiv j\pmod h.
\end{equation}

Assume first that $q>5$. If $p\ge7$, then for every $0\le i\le a-1$ and $j\in\{1,2\}$,
\[
 0<3p^i-j\le3p^{a-1}-1<\frac{p^a-1}{2}\le h,
\]
so \eqref{eq:covariance} is impossible. Since $p\equiv2\pmod3$, it remains to consider $p=5$. Then $a\ge3$. For $i\le a-2$ the same inequality applies. For $i=a-1$, the congruence is impossible if $h=q-1$. If $h=(q-1)/2$, divisibility would force
\[
 3\cdot5^{a-1}-j=\frac{5^a-1}{2},
\]
and hence
\[
 5^{a-1}=2j-1\in\{1,3\},
\]
again impossible. Thus $L_1=L_2=0$, and
\[
 V=\F_q u^3.
\]
The corresponding extension is
\[
 v^q-v=u^3,
\]
which is the required curve after interchanging the coordinate names.

It remains to consider $q=5$. In the $\PGL(2,5)$ case, \eqref{eq:covariance} again gives $L_1=L_2=0$. In the $\PSL(2,5)$ case, the only possible lower term is linear, so
\begin{equation}\label{eq:genus4family}
 v^5-v=u^3+cu
\end{equation}
for some $c\in k$.

We show that $c=0$. The map defined by $v$ has degree $3$. Let $P_\infty$ be the unique point above infinity for the Artin--Schreier map $u:X\to\PP^1$. Then
\[
 \operatorname{ord}_{P_\infty}(u)=-5,
 \qquad
 \operatorname{ord}_{P_\infty}(v)=-3.
\]
The different exponent of the degree-$5$ cover is $(5-1)(3+1)=16$, and hence
\[
 (du)=6P_\infty.
\]
Therefore
\[
 du,
 \quad v\,du,
 \quad v^2du,
 \quad u\,du
\]
form a basis of the canonical space. In the associated canonical coordinates one has
\[
 Z_0Z_2=Z_1^2.
\]
The curve is not hyperelliptic: otherwise its degree-$2$ map and the degree-$3$ map defined by $v$ would be independent, and Castelnuovo--Severi would give $g\le2$. Hence the canonical model is contained in a unique quadric, and the displayed quadric is a cone. Its ruling gives the unique trigonal pencil. This holds for every $c\in k$, including $c=0$, which proves the last assertion of the lemma.

If $c\ne0$, the finite ramification points of the degree-$3$ map $v$ are given by
\[
 3u^2+c=0
\]
and have ramification index $2$, whereas $P_\infty$ has index $3$. Thus $P_\infty$ is the unique ramification point of index $3$ of the unique trigonal pencil. Every automorphism of $X$ must fix it, contradicting the hypothesis on $G$. Therefore $c=0$, and \eqref{eq:genus4family} is again the curve $y^3=x^5-x$.
\end{proof}

The curve in Lemma~\ref{lem:cubic} has the tame deck transformation
\[
 (x,y)\longmapsto(x,\zeta_3 y),
\]
which fixes every point lying over $\PP^1(\F_q)$ and hence acts trivially on $\Omega$. Therefore the action of the full automorphism group on $\Omega$ has a nontrivial kernel of order $3$, contradicting $K=1$. Thus the linear case does not occur in the strip.

\subsection{The unitary case}

Suppose that a subgroup
\[
 U\cong\PSU(3,n)
\]
acts naturally on $\Omega$, where $q=n^3$ and $n$ is odd. Put
\[
 \mu=\gcd(3,n+1).
\]
The group $U$ has no global fixed point on $X$, because a point stabilizer is solvable. Hence Proposition~\ref{prop:orbits-subgroup} applies to $U$. If $w_U$ denotes the stabilizer order in its tame short orbit, with $w_U=1$ when no such orbit exists, then
\begin{equation}\label{eq:unitary}
 2g=
 \frac{(n^3+1)(n^2-1)}{\mu w_U}
 -(n^3-1).
\end{equation}
By \cite[Lemma~2.6(i)]{GK10}, $w_U$ is odd and
\[
 w_U\mid n+1
 \quad\text{or}\quad
 w_U\mid\frac{n^2-n+1}{\mu}.
\]

If $w_U\mid n+1$, then $w_U\le(n+1)/2$. Combining \eqref{eq:unitary} with $g<n^3$ gives
\[
 \frac{2(n-1)}{\mu}<3.
\]

If $\mu=1$ this forces $n<5/2$, which is impossible. Hence $\mu=3$, so $n<11/2$ and $3\mid n+1$, which gives $n=5$. Then $w_U\in\{1,3\}$. For $w_U=1$, \eqref{eq:unitary} gives $g=442>n^3=125$. For $w_U=3$ it gives $g=106$, which lies in the strip $[100,125)$.

Thus the only possibility compatible with the strip is
\begin{equation}\label{eq:106}
 n=5,\qquad w_U=3,\qquad g=106.
\end{equation}

If
\[
 w_U=\frac{n^2-n+1}{\mu r},
\]
with $r$ odd, then
\begin{equation}\label{eq:unitary-r}
 2g=(n-1)\left(r(n+1)^2-(n^2+n+1)\right).
\end{equation}
For $r\ge3$ the right-hand side is larger than $2n^3$, contrary to $g<q=n^3$. For $r=1$ one gets
\[
 g=\frac{n(n-1)}2
 <
 \frac{p-1}{p}n^3,
\]
contrary to the lower bound in \eqref{eq:strip}. Thus only \eqref{eq:106} remains. The same reduction treats $\PGU(3,n)$ by passing to its subgroup $\PSU(3,n)$.

\begin{lemma}\label{lem:106}
The case
\[
 n=5,
 \qquad w_U=3,
 \qquad g=106
\]
does not occur.
\end{lemma}

\begin{proof}
A Sylow $5$-subgroup $S$ of $\PSU(3,5)$ is extraspecial of order $125$, and its natural point stabilizer has structure
\[
 5^{1+2}:8;
\]
see \cite{Atlas}. Put
\[
 Z=Z(S)=[S,S].
\]
Then
\[
 |Z|=5,
 \qquad
 S/Z\cong(\F_5)^2,
\]
and a cyclic complement $H\cong C_8$ acts faithfully and irreducibly on $S/Z$.

Explicitly, use the form and the torus $h(\lambda)$ of Lemma~\ref{lem:nofield}. The Sylow subgroup fixing $\langle e_1\rangle$ consists of the matrices
\[
 u(\alpha,\beta)=\begin{pmatrix}1&\alpha&\beta\\0&1&-\bar\alpha\\0&0&1\end{pmatrix},
 \qquad \beta+\bar\beta+\alpha\bar\alpha=0,
\]
with $Z=\{u(0,\beta)\}$ and $S/Z\cong(\F_{25},+)$ via $\alpha$. One has
\[
 h(\lambda)\,u(\alpha,\beta)\,h(\lambda)^{-1}=u(\lambda^{2-n}\alpha,\lambda^{n+1}\beta).
\]
For $n=5$, the image of $\lambda\mapsto\lambda^{-3}$ is the subgroup of order $8$ of $\F_{25}^*$, and its kernel is the group of scalars. Hence $H\cong C_8$ acts on $S/Z\cong\F_{25}$ by multiplication by a primitive $8$th root of unity. Since $8\nmid5-1$, this action is faithful and $\F_5$-irreducible.

Every lower ramification group is normal in $S\rtimes H$. Let $R$ be one of these groups. The image of $R$ in $S/Z$ is $H$-invariant, hence is either $0$ or all of $S/Z$. If the image is $0$, then $R\le Z$, so $R=Z$ or $1$. If the image is all of $S/Z$, then
\[
 RZ=S.
\]
Since $S$ is extraspecial, its Frattini subgroup is
\[
 \PhiG(S)=Z.
\]
Thus a subgroup whose image generates $S/\PhiG(S)$ must be all of $S$, and therefore $R=S$. Hence the only possible nontrivial ramification groups are
\[
 S\quad\text{and}\quad Z.
\]
The central stage cannot be skipped: if $S_a=S$ and $S_{a+1}=1$, then
\[
 [S_a,S_a]\subseteq S_{2a+1}\subseteq S_{a+1}=1
\]
by \cite[Chapter~IV, Proposition~10]{Serre}, contrary to $[S,S]=Z\ne1$.

Consequently there are exactly two positive jumps $a<b$ such that
\[
 S_i=
 \begin{cases}
 S,&1\le i\le a,\\
 Z,&a<i\le b,\\
 1,&i>b.
 \end{cases}
\]
Equation \eqref{eq:genusS} gives
\[
 212=124(a-1)+4(b-a),
\]
so
\[
 b=84-30a.
\]
The conditions $1\le a<b$ leave only
\[
 (a,b)=(1,54)
 \quad\text{or}\quad
 (2,24).
\]

Suppose first that $(a,b)=(2,24)$. Then
\[
 S_2/S_3\cong S/Z.
\]
Let $\chi:H\to k^*$ be the tangent character at $P$. It is faithful and has order $8$. The action of a tame complement on the $i$th lower ramification quotient is given, up to inversion, by the $i$th power of the tangent character; see \cite[Chapter~IV]{Serre}. Hence the image of $H$ on $S_2/S_3$ has order dividing
\[
 \frac{8}{\gcd(8,2)}=4.
\]
This contradicts the faithful action of $H$ on $S/Z$.

Now suppose $(a,b)=(1,54)$. The subgroup $Z$ has a single lower jump $54$. Hurwitz for $X\to X/Z$ gives
\[
 210=5\bigl(2g(X/Z)-2\bigr)+4\cdot55,
\]
so
\[
 X/Z\cong\PP^1.
\]
Choose a coordinate $t$ for which the branch point is infinity. Artin--Schreier theory gives
\[
 k(X)=k(t,z),
 \qquad
 z^5-z=f(t),
\]
where $f$ is reduced and
\[
 \deg f=54.
\]

The group $S/Z\cong(\F_5)^2$ acts on the $t$-line as a group of translations. Choose a nonzero translation
\[
 t\longmapsto t+c.
\]
Since $Z$ is central in $S$, a lift commutes with the deck transformations and has the form
\[
 (t,z)\longmapsto(t+c,z+w(t))
\]
for some $w(t)\in k(t)$. Therefore
\begin{equation}\label{eq:AS106}
 f(t+c)-f(t)=w(t)^5-w(t).
\end{equation}
The left-hand side has degree exactly $53$, because its leading coefficient is
\[
 54c\,\operatorname{lc}(f)\ne0
\]
in characteristic $5$. The right-hand side has no finite poles, so $w$ is a polynomial. If $w$ is nonconstant, its degree is a multiple of $5$; if $w$ is constant, the right-hand side is constant. Both alternatives contradict \eqref{eq:AS106}. Thus the case \eqref{eq:106} is impossible.
\end{proof}

Hence the unitary case is impossible in the strip.

\subsection{The Ree case}

Suppose
\[
 G\cong\Ree(n),
 \qquad
 n=3n_0^2,\qquad n_0=3^a,
\]
so that $q=n^3$ and the complement in a natural point stabilizer has order $n-1$. Equation \eqref{eq:master2} becomes
\begin{equation}\label{eq:ree}
 2g=
 \frac{(n^3+1)(n-1)}w
 -(n^3-1).
\end{equation}
By \cite[Lemma~2.6(iii)]{GK10},
\[
 w\mid n+1,
 \quad\text{or}\quad
 w\mid n+3n_0+1,
 \quad\text{or}\quad
 w\mid n-3n_0+1,
\]
and $w$ is odd.

For $n=3$, the only positive genus at least two produced by these divisibilities is $g=15$, while the strip begins at
\[
 \frac23q=18.
\]
Thus $n=3$ does not occur.

Assume $n\ge27$. If
\[
 w=\frac{n+1}{r},
\]
then $r\ge4$ and
\[
 2g\ge(n-1)(3n^2-5n+3)>2n^3,
\]
contrary to $g<q$. If
\[
 w=\frac{n+3n_0+1}{r},
\]
then either the genus is negative for $r=1$, or $r\ge5$ and
\[
 2g\ge\frac{n-1}{3}(7n^2+22n+12)>2n^3.
\]
Finally, if
\[
 w=\frac{n-3n_0+1}{r},
\]
then $r>1$ gives an even larger genus than in the preceding branch. For $r=1$,
\[
 g=\frac32n_0(n-1)(n+n_0+1).
\]
Since $n_0+1<n/3$ for $n_0\ge3$,
\[
 g<2n_0n^2<\frac23n^3,
\]
because $n_0<n/3$. This contradicts the lower bound in \eqref{eq:strip}. Thus the Ree case is impossible.

We have proved that, under \eqref{eq:strip}, the faithful action $K=1$ cannot occur.

\section{The cubic family and completion of the proof}

We now determine the full automorphism group of the additional family.

\begin{proposition}\label{prop:cubic-aut}
Let
\[
 C_q:\ y^3=x^q-x,
 \qquad
 q=p^a\equiv2\pmod3.
\]
Then
\[
 g(C_q)=q-1,
 \qquad
 \gamma(C_q)=0,
\]
and
\begin{equation}\label{eq:cubic-aut}
 \Aut_k(C_q)\cong C_3\times\PGL(2,q).
\end{equation}
A Sylow $p$-subgroup has order $q$, and the full automorphism group has no global fixed point.
\end{proposition}

\begin{proof}
The degree-$3$ map to the $x$-line is totally ramified at the $q+1$ points of $\PP^1(\F_q)$, so Hurwitz gives
\[
 g=q-1.
\]
The translations
\[
 (x,y)\longmapsto(x+b,y),
 \qquad b\in\F_q,
\]
form a group of order $q$ with a unique fixed point. Its quotient is rational and is ramified only at that point, so Deuring--Shafarevich gives $\gamma=0$.

For $q>5$, the degree-$3$ pencil is unique by Castelnuovo--Severi, since two distinct degree-$3$ maps would imply $g\le4$. For $q=5$, uniqueness was proved in Lemma~\ref{lem:cubic} (last assertion). Hence every automorphism descends to the $x$-line and preserves the branch set $\PP^1(\F_q)$. The stabilizer of this set in $\PGL(2,k)$ is $\PGL(2,q)$, and the kernel of the induced action is the deck group $C_3$. Therefore
\[
 |\Aut_k(C_q)|\le3|\PGL(2,q)|.
\]

Conversely, every element of $\PGL(2,q)$ lifts. Put
\[
 r=\frac{2q-1}{3},
 \qquad
 d=\frac{q+1}{3}.
\]
For
\[
 A=\begin{pmatrix}a&b\\ c&e\end{pmatrix}\in\operatorname{GL}(2,q),
\]
define
\[
 (x,y)\longmapsto
 \left(
 \frac{ax+b}{cx+e},
 \frac{\det(A)^r}{(cx+e)^d}y
 \right).
\]
Since
\[
 \left(\frac{ax+b}{cx+e}\right)^q-
 \frac{ax+b}{cx+e}
 =
 \frac{\det(A)(x^q-x)}{(cx+e)^{q+1}},
\]
and
\[
 3r=2(q-1)+1,
\]
this transformation preserves the equation. A scalar matrix $\lambda I$ acts trivially because
\[
 2r-d=q-1.
\]
Thus the lifts give a subgroup isomorphic to $\PGL(2,q)$. It commutes with the deck group and meets it trivially, proving \eqref{eq:cubic-aut}.

The $p$-part of $3|\PGL(2,q)|$ is $q$, so the displayed translation group is Sylow. Its common fixed point is moved by the lifted projective transformations. Hence the full group has no global fixed point.
\end{proof}

\begin{proof}[Proof of Theorem~\ref{thm:zerorank}]
Let $S^*$ be a Sylow $p$-subgroup of $G$ containing the subgroup $S$ from the statement. By Lemma~\ref{lem:onepoint}, $S^*$ has a unique fixed point. That point is fixed by $S$, hence it is the unique fixed point $P$ of $S$. Thus $S^*\le G_P$ and $|S^*|\ge |S|>g$. Replacing $S$ by $S^*$, we may assume for the proof that $S$ is Sylow.

If the strict inequality
\[
 |S|>\frac{p}{p-1}g
\]
holds, the three classical families follow from \cite[Theorem~1.1]{GK10}. Otherwise we are in the strip \eqref{eq:strip}. Section~5 shows that the faithful action $K=1$ is impossible there, while Proposition~\ref{prop:kernel} shows that the nonfaithful case gives either the classical hyperelliptic family or the cubic family. Proposition~\ref{prop:cubic-aut} determines the full group of the latter. In that case the Sylow order is $g+1$; since the original subgroup in Theorem~\ref{thm:zerorank} has order $>g$ and is contained in this Sylow subgroup, it must itself have order $g+1$.

Conversely, the first three families have the stated Sylow orders and fixed-point properties by \cite[Introduction and Theorem~1.1]{GK10}. Their Sylow $p$-subgroups have rational quotient and a unique ramified point, so Deuring--Shafarevich gives $p$-rank zero. Proposition~\ref{prop:cubic-aut} gives the same assertions for the fourth family.
\end{proof}

\begin{proof}[Proof of Corollary~\ref{cor:zerogap}]
Only the fourth family of Theorem~\ref{thm:zerorank} is nonclassical, and there
\[
 |S|=q=g+1.
\]
This proves the assertion.
\end{proof}

\section{The equality case}

We now remove the zero-$p$-rank hypothesis and assume throughout this section that
\begin{equation}\label{eq:equality}
 |S|=\frac{p}{p-1}g.
\end{equation}
Write $q=|S|=p^a$. Thus $g=(p-1)q/p$.

\begin{proposition}\label{prop:rank-dichotomy}
Assume \eqref{eq:equality}. Then either $X$ has $p$-rank zero, or $X$ is ordinary. In the ordinary case $X/S$ is rational, the action of $S$ has exactly two short orbits, namely a fixed point $P$ and an orbit $\Delta$ of length $q/p$, each point of $\Delta$ has stabilizer of order $p$, and the second ramification groups of $S$ at $P$ and of $S_R$ at $R\in\Delta$ are trivial.
\end{proposition}

\begin{proof}
Let $\gamma$ be the $p$-rank of $X$ and $\bar\gamma$ the $p$-rank of $X/S$. Besides $\{P\}$, let the short $S$-orbits have lengths $\ell_1,\ldots,\ell_r$. Deuring--Shafarevich gives
\[
 \gamma=q\bar\gamma+\sum_{i=1}^r(q-\ell_i).
\]
Since $\gamma\le g<q$, we have $\bar\gamma=0$. Each nontrivial stabilizer has order at least $p$, so $\ell_i\le q/p$ and
\[
 q-\ell_i\ge q-q/p=g.
\]
Thus either $r=0$ and $\gamma=0$, or $r=1$, $\ell_1=q/p$, and $\gamma=g$.

Assume the latter. For $R\in\Delta$, $|S_R|=p$. Hurwitz gives
\[
 2g-2=q(2\bar g-2)+d_P+\frac qp d_R.
\]
Since $d_P\ge2(q-1)$ and $d_R\ge2(p-1)$,
\[
 2g-2\ge2q\bar g+2g-2.
\]
Hence $\bar g=0$ and equality holds in the two lower bounds. Therefore $S_P^{(2)}=1$ and $(S_R)^{(2)}=1$.
\end{proof}

\begin{remark*}
For $q=p$, Proposition~\ref{prop:rank-dichotomy} can also be deduced from \cite[Theorem~1.1]{GK17}; see the Remark in the Introduction.
\end{remark*}

\begin{lemma}\label{lem:S-sylow-boundary}
Assume \eqref{eq:equality} and suppose that $X$ is not one of the classical curves in \cite[Theorem~1.1]{GK10}. Then $S=G_P^{(1)}$ and $S$ is a Sylow $p$-subgroup of $G$.
\end{lemma}

\begin{proof}
If $S<G_P^{(1)}$, then $|G_P^{(1)}|>pg/(p-1)$, so \cite[Theorem~1.1]{GK10} makes $X$ classical. Hence $S=G_P^{(1)}$.

Let $S^*$ be a Sylow $p$-subgroup of $G$ containing $S$ and suppose $S<S^*$. The normalizer condition gives $S<N_{S^*}(S)$. If the $p$-rank is zero, or if the curve is ordinary with $q>p$, the point $P$ is the unique common fixed point of $S$, so $N_{S^*}(S)$ fixes $P$, contradicting $S=G_P^{(1)}$. If $q=p$ in the ordinary case, $S$ has exactly two fixed points. Since $p$ is odd, the $p$-group $N_{S^*}(S)$ fixes both of them, giving the same contradiction. Thus $S$ is Sylow.
\end{proof}

\begin{corollary}\label{cor:zero-boundary}
Assume \eqref{eq:equality}, suppose that $X$ has $p$-rank zero, and assume that $X$ is not classical. Then $q=p$, $p\equiv2\pmod3$, and
\[
 X\cong\{y^3=x^p-x\}.
\]
\end{corollary}

\begin{proof}
By Lemma~\ref{lem:S-sylow-boundary}, $S$ is Sylow. Since $q>g$, Theorem~\ref{thm:zerorank} applies. The only nonclassical possibility is $y^3=x^q-x$, with $q\equiv2\pmod3$ and genus $q-1$. Comparing $q-1=(p-1)q/p$ gives $q=p$.
\end{proof}

\section{The ordinary boundary}

Assume the equality case is ordinary. Since $S=S_P^{(1)}$ and $S_P^{(2)}=1$, the group $S$ is elementary abelian.

\begin{proposition}\label{prop:ordinary-p}
If $q=p$, then $g=p-1$ and, after a change of coordinates,
\begin{equation}\label{eq:ordinary-two-poles}
 X:\qquad y^p-y=x+\frac{c}{x},\qquad c\in k^*.
\end{equation}
Conversely, every curve \eqref{eq:ordinary-two-poles} is ordinary of genus $p-1$, realizes equality, and its full automorphism group has no global fixed point.
\end{proposition}

\begin{proof}
The second short orbit has length one, so $X\to X/S\cong\PP^1$ is weakly ramified at exactly two points. Sending them to $0$ and $\infty$, reduced Artin--Schreier theory gives $y^p-y=ax+b/x$ with $ab\ne0$. Replacing $x$ by $ax$ gives \eqref{eq:ordinary-two-poles}.

Conversely, the right-hand side has two simple poles, so Hurwitz and Deuring--Shafarevich give $g=\gamma=p-1$. The translations $y\mapsto y+r$, $r\in\F_p$, form $C_p$, while $(x,y)\mapsto(c/x,y)$ interchanges the two ramified points. Hence the full automorphism group has no global fixed point.
\end{proof}

Giulietti and Korchm\'aros \cite[Remarks~2.2--2.3]{GK17} discuss the curves \eqref{eq:ordinary-two-poles} and their automorphism groups, including the exceptional curve in characteristic $3$ with group $\GL(2,3)$. For completeness we give a self-contained proof, and we add the isomorphism criterion. These curves attain Nakajima's bound $|S|\le\frac{p}{p-2}(\gamma-1)$ \cite{Na87} with equality, since $|S|=p$ and $\gamma=p-1$.

\begin{proposition}[Automorphisms and moduli of the Kloosterman boundary family]\label{prop:kloosterman-aut}
Let
\[
 C_c:\qquad y^p-y=x+\frac c x,\qquad c\in k^*,
\]
with $p$ odd.

If $p\ge5$, then
\[
 \Aut_k(C_c)\cong C_2\times D_{2p},
 \qquad
 |\Aut_k(C_c)|=4p,
\]
where $D_{2p}$ denotes the dihedral group of order $2p$. Moreover
\[
 C_c\cong C_{c'}
 \quad\Longleftrightarrow\quad
 \frac{c'}c\in(\F_p^*)^2.
\]

If $p=3$, then
\[
 \Aut_k(C_c)\cong C_2\times D_6
 \quad\text{for }c\ne-1,
\]
whereas
\[
 \Aut_k(C_{-1})\cong\GL(2,3).
\]
Thus $|\Aut_k(C_c)|=12$ for $c\ne-1$ and $|\Aut_k(C_{-1})|=48$. In characteristic $3$,
\[
 C_c\cong C_{c'}
 \quad\Longleftrightarrow\quad c=c'.
\]
\end{proposition}

\begin{proof}
Put
\[
 u=y^p-y,\qquad z=2x-u.
\]
Then
\[
 z^2=u^2-4c=(y^p-y)^2-4c.
\]
Thus $C_c$ is hyperelliptic of genus $p-1$, with hyperelliptic involution
\[
 \iota:(x,y)\longmapsto(c/x,y).
\]
The transformations
\[
 \tau_r:(x,y)\longmapsto(x,y+r),\qquad r\in\F_p,
\]
and
\[
 \sigma:(x,y)\longmapsto(-x,-y)
\]
generate a dihedral group $D_{2p}$, and $\iota$ commutes with it. Hence
\[
 C_2\times D_{2p}\le\Aut_k(C_c).
\]

Let
\[
 \bar G=\Aut_k(C_c)/\langle\iota\rangle.
\]
On the hyperelliptic $y$-line it preserves
\[
 B_c=\{t\in k:(t^p-t)^2=4c\},
 \qquad |B_c|=2p.
\]
Choose $\alpha\in k$ with
\[
 (\alpha^p-\alpha)^2=4c.
\]
Then
\[
 B_c=(\alpha+\F_p)\cup(-\alpha+\F_p),
\]
with the two cosets distinct. The reduced group contains
\[
 U=\{t\mapsto t+r:r\in\F_p\}
\]
and the involution $t\mapsto-t$; together they form $D_{2p}$ and act transitively on $B_c$.

A $p$-subgroup of $\bar G$ cannot have order larger than $p$. Indeed, a nontrivial finite $p$-subgroup of $\PGL(2,k)$ has a unique fixed point and is semiregular away from it. If its fixed point lies in $B_c$, the remaining $2p-1$ points cannot be partitioned into free orbits of size at least $p^2$; if it lies outside $B_c$, its order must divide $2p$. Thus $U$ is Sylow in $\bar G$.

Suppose first that $\bar G$ fixes infinity. Then $U$ is the unique Sylow $p$-subgroup. Hence every element of $\bar G$ normalizes $U$ and has the form
\[
 t\longmapsto at+b,
 \qquad a\in\F_p^*,\quad b\in k.
\]
On the quotient additive group $k/\F_p$, preservation of the unordered pair
\[
 \{\bar\alpha,-\bar\alpha\}
\]
forces $\bar b=0$ and $a=\pm1$. Hence
\[
 \bar G=D_{2p}.
\]

Assume now that $\bar G$ moves infinity. By the classification of finite $p$-irregular subgroups of $\PGL(2,k)$, the $p$-semi-elementary case is impossible because its Sylow $p$-subgroup is normal and would fix infinity. For $p\ge7$, the only remaining possibilities are $\PSL(2,p)$ and $\PGL(2,p)$. But
\[
 N_{\bar G}(U)
\]
is contained in the affine stabilizer of $B_c$. The preceding affine calculation shows that this stabilizer is exactly $D_{2p}$. On the other hand, the normalizer of a Sylow $p$-subgroup has order
\[
 \frac{p(p-1)}2
 \quad\text{in }\PSL(2,p),
 \qquad
 p(p-1)
 \quad\text{in }\PGL(2,p).
\]
For $p\ge7$ neither can be contained in a group of order $2p$. Hence no larger reduced group exists.

For $p=5$, the only additional possibility is
\[
 \bar G\cong\PSL(2,5)\cong A_5.
\]
Since $D_{10}$ is transitive on the ten branch points, $\bar G$ would also be transitive, and a branch-point stabilizer would have order $60/10=6$. Such a stabilizer is tame and therefore cyclic, but $A_5$ has no cyclic subgroup of order $6$. Hence again
\[
 \bar G=D_{10}.
\]
This proves, for every $p\ge5$,
\[
 \Aut_k(C_c)=\langle\iota\rangle\times D_{2p}.
\]

Now let $p=3$. The group $\PSL(2,3)\cong A_4$ is excluded: $\bar G$ contains $D_6$, whereas $A_4$ has no subgroup of order $6$. If $\bar G$ moves infinity, the only possible larger reduced group is
\[
 \PGL(2,3)\cong S_4.
\]
(The exceptional $A_5$ would have a branch-point stabilizer of order $10$, impossible because a tame point stabilizer is cyclic and $A_5$ has no cyclic subgroup of order $10$.) We show that this occurs precisely for $c=-1$.

Assume first that $\bar G\cong\PGL(2,3)$. All subgroups isomorphic to $\PGL(2,3)$ are conjugate in $\PGL(2,k)$ \cite{Faber}. After conjugating and then composing with an element of the standard $\PGL(2,3)$, we may arrange that the Sylow subgroup $U$ is the standard translation group. The conjugating projectivity then normalizes $U$, hence is affine. Since it also carries
\[
 N_{\bar G}(U)=D_6
\]
onto the standard normalizer of $U$, it belongs to $D_6$. Thus $\bar G$ itself is the standard $\PGL(2,3)$.

The standard group has a unique orbit of length $6$ on $\PP^1(k)$; it is
\[
 \PP^1(\F_9)\setminus\PP^1(\F_3)=\F_9\setminus\F_3.
\]
Indeed a point in such an orbit has cyclic stabilizer of order $4$, and all cyclic subgroups of order $4$ in $S_4$ are conjugate. Therefore
\[
 B_c=\F_9\setminus\F_3.
\]
For $t\in\F_9\setminus\F_3$, the nonzero trace-zero element $t^3-t$ has square $-1$. Since $4=1$ in characteristic $3$, the equation defining $B_c$ gives
\[
 c=-1.
\]

Conversely, if $c=-1$, then
\[
 B_{-1}=\F_9\setminus\F_3,
\]
so the reduced group is $\PGL(2,3)$. In this case the hyperelliptic model is
\[
 z^2=P(t),
 \qquad
 P(t)=t^6+t^4+t^2+1
      =\frac{t^9-t}{t^3-t}.
\]
For
\[
 A=\begin{pmatrix}a&b\\ c&d\end{pmatrix}\in\GL(2,3),
 \qquad
 t'=\frac{at+b}{ct+d},
\]
one has
\[
 P(t')=\frac{P(t)}{(ct+d)^6}.
\]
Hence
\[
 (t,z)\longmapsto
 \left(t',\frac{z}{(ct+d)^3}\right)
\]
gives a faithful action of $\GL(2,3)$; the scalar matrix $-I$ acts as the hyperelliptic involution. Since the reduced group has order $24$, this already has the maximal possible order $48$. Thus
\[
 \Aut_k(C_{-1})\cong\GL(2,3).
\]
For $c\ne-1$, the reduced group is $D_6$, and the explicit lifts above give
\[
 \Aut_k(C_c)\cong C_2\times D_6.
\]

It remains to determine the parameter. Let first $p\ge5$ and suppose
\[
 C_c\cong C_{c'}.
\]
The hyperelliptic involution is characteristic, and the reduced group $D_{2p}$ has a unique Sylow $p$-subgroup. Hence the isomorphism carries the Artin--Schreier deck group to the corresponding deck group and descends to the $x$-line preserving the unordered branch pair $\{0,\infty\}$. Thus it has the form
\[
 x\longmapsto\lambda x
 \quad\text{or}\quad
 x\longmapsto\frac{\lambda}{x}.
\]
Its conjugation action on the deck group is multiplication by some $a\in\F_p^*$. Comparing the reduced Artin--Schreier equations gives, in either case,
\[
 c'=a^2c.
\]
Conversely, if $c'=a^2c$ with $a\in\F_p^*$, then
\[
 (x,y)\longmapsto(ax,ay)
\]
is an isomorphism $C_c\to C_{c'}$.

For $p=3$, the special curve $C_{-1}$ is distinguished by its automorphism-group order. For the remaining curves the reduced group is $D_6$, whose Sylow $3$-subgroup is unique, so the same argument applies. Since
\[
 (\F_3^*)^2=\{1\},
\]
one obtains
\[
 C_c\cong C_{c'}
 \quad\Longleftrightarrow\quad
 c=c'.
\]
\end{proof}

Proposition~\ref{prop:kloosterman-aut} completes the automorphism-group and parameter classification of the genus-$p-1$ ordinary boundary family.

Assume now $q>p$ and put $n=q/p=p^{a-1}\ge p$. Fix $R\in\Delta$ and set $L=S_R$, so $|L|=p$. Since $S$ is abelian, $L$ fixes every point of $\Delta$ and $P$.

\begin{lemma}\label{lem:Lquotient}
One has $\operatorname{Fix}(L)=\{P\}\cup\Delta$ and $X/L\cong\PP^1$. After choosing a coordinate $x$ on $X/L$, there is an additive subgroup $W\subset k$ of order $n$ such that
\[
 S/L=\{x\mapsto x+w:w\in W\},
\]
and the branch set of $X\to X/L$ is $B=W\cup\{\infty\}$.
\end{lemma}

\begin{proof}
Any point fixed by a nontrivial element of $L$ lies in a short $S$-orbit, so Proposition~\ref{prop:rank-dichotomy} gives the fixed set. The group $L$ is weakly ramified at its $n+1$ fixed points. Hurwitz gives
\[
 2g-2=p(2g(X/L)-2)+2(n+1)(p-1),
\]
and $g=n(p-1)$ yields $g(X/L)=0$. The quotient $S/L$ acts faithfully on the line, fixes the image of $P$, and acts regularly on the images of $\Delta$. Conjugating in $\PGL(2,k)$ gives the translation form.
\end{proof}

Let $A_W(x)=\prod_{w\in W}(x-w)$. This is a separable additive polynomial and $A_W'$ is a nonzero constant.

\begin{lemma}\label{lem:ASnormalform}
The curve has an Artin--Schreier equation
\begin{equation}\label{eq:ASnormalform}
 y^p-y=\alpha x+\frac{c}{A_W(x)},\qquad \alpha c\ne0.
\end{equation}
\end{lemma}

\begin{proof}
Weak ramification over $B$ gives a reduced equation
\[
 y^p-y=\alpha x+\sum_{w\in W}\frac{c_w}{x-w},
\]
with all displayed coefficients nonzero. For $v\in W$, a lift in the abelian group $S$ of $x\mapsto x+v$ commutes with the deck group and has the form $(x,y)\mapsto(x+v,y+h_v(x))$. Hence
\[
 f(x+v)-f(x)=h_v(x)^p-h_v(x).
\]
The left side has only possible simple poles, while every pole of a nonzero Artin--Schreier coboundary has order divisible by $p$. Thus it has no finite poles, and comparison of principal parts gives $c_{w+v}=c_w$. Hence all $c_w$ are equal. Since
\[
 \sum_{w\in W}\frac1{x-w}=\frac{A_W'}{A_W},
\]
the result follows after absorbing the nonzero constant $A_W'$ into $c$.
\end{proof}

\begin{lemma}\label{lem:Lnormal}
If $q>p$, then $L\triangleleft G$.
\end{lemma}

\begin{proof}
If $L^g\ne L$ for some $g\in G$, then $X/L$ and $X/L^g$ are two distinct rational degree-$p$ quotients. Their function fields are distinct subfields of index $p$, hence their compositum is $k(X)$. Castelnuovo--Severi gives $g\le(p-1)^2$. But $g=n(p-1)\ge p(p-1)>(p-1)^2$, a contradiction.
\end{proof}

Put $H=G/L\le\PGL(2,k)$. It preserves $B$, and $U=S/L$ is a Sylow $p$-subgroup of $H$ of order $n$. The group $H$ does not fix $\infty$, since otherwise $G$ would fix the unique point $P$ above it. Hence $H$ is transitive on $B$, and because $U\le H_\infty$ acts regularly on $B\setminus\{\infty\}$, the action is doubly transitive.

\begin{lemma}\label{lem:reduced-boundary}
After an affine change of coordinate fixing $\infty$,
\[
 W=\F_n,\qquad B=\PP^1(\F_n),
\]
and $H\cong\PSL(2,n)$ or $\PGL(2,n)$.
\end{lemma}

\begin{proof}
By Faber's classification \cite{Faber}, a finite $p$-irregular subgroup of $\PGL(2,k)$ is linear, $p$-semi-elementary, dihedral, or exceptional $A_5$. A $p$-semi-elementary group fixes a unique point, contrary to the fact that $H$ moves $\infty$; dihedral $p$-irregular groups occur only in characteristic $2$. The exceptional $A_5$ in characteristic $3$ has Sylow $3$-subgroups of order $3$, and it has no subgroup of index $4$, so it cannot have the required orbit when $n=3$. Thus $H$ is $\PSL(2,p^e)$ or $\PGL(2,p^e)$. Comparing Sylow $p$-orders gives $p^e=n$. By conjugacy of the linear groups, and then composing with an element of $\PGL(2,n)$ sending the fixed point of the Sylow subgroup back to $\infty$, we may take the conjugating projectivity to fix $\infty$. It is therefore affine, and the orbit becomes $\PP^1(\F_n)$ with translation subgroup $\F_n$.
\end{proof}

\begin{proposition}\label{prop:ordinary-exception}
If $q>p$, then necessarily $p=3$, $q=9$, $g=6$, and, after a change of coordinates,
\begin{equation}\label{eq:genus6}
 X:\qquad y^3-y=c\left(x+\frac1{x^3-x}\right),\qquad c\in k^*.
\end{equation}
Conversely, every curve \eqref{eq:genus6} realizes equality and its full automorphism group has no global fixed point.
\end{proposition}

\begin{proof}
By Lemmas~\ref{lem:ASnormalform} and~\ref{lem:reduced-boundary},
\[
 y^p-y=f(x),\qquad f(x)=\alpha x+\frac{c}{x^n-x}.
\]
The element $\iota:x\mapsto-1/x$ belongs to $\PSL(2,n)\le H$. Choose a lift to $G$. Its conjugation action on $L\cong C_p$ is multiplication by $\lambda\in\F_p^*$, so the lift has the form $y\mapsto\lambda y+h(x)$ and
\[
 f(-1/x)-\lambda f(x)=h(x)^p-h(x).
\]
Using
\[
 (-1/x)^n-(-1/x)=\frac{x^n-x}{x^{n+1}},\qquad
 \frac{x^{n+1}}{x^n-x}=x+\frac{x^2}{x^n-x},
\]
the left side becomes
\[
 -\frac{\alpha}{x}+(c-\lambda\alpha)x+c\frac{x^2-\lambda}{x^n-x}.
\]
An Artin--Schreier coboundary has no simple poles. At each $w\in\F_n^*$ the polynomial $x^n-x$ has a simple zero with derivative $-1$, so the residue there is $-c(w^2-\lambda)$, and we therefore obtain $w^2=\lambda$. Hence $n-1\le2$. Since $n\ge p\ge3$, one has $n=p=3$. Then $\lambda=1$, and cancellation at $x=0$ gives $c=\alpha$, proving \eqref{eq:genus6}.

Conversely let $F(x)=c(x+1/(x^3-x))$. In characteristic $3$, $F(x+t)=F(x)+ct$ for $t\in\F_3$ and $F(-1/x)=F(x)$. Choose $u$ with $u^3-u=c$. The maps
\[
 (x,y)\mapsto(x+t,y+tu+r),\qquad t,r\in\F_3,
\]
form a group $S\cong C_3^2$ of order $9$. The right side has four simple poles, so $g=\gamma=6$ by Hurwitz and Deuring--Shafarevich. The deck subgroup $\{(x,y)\mapsto(x,y+r):r\in\F_3\}$ fixes exactly the four ramification points, while the translations $x\mapsto x+t$ act transitively on the three finite ones and fix the point above $\infty$. Hence the latter point is the unique common fixed point of $S$. The involution $(x,y)\mapsto(-1/x,y)$ moves it, so the full automorphism group has no global fixed point.
\end{proof}

\begin{proposition}[Full group and parameter for the genus-six boundary family]\label{prop:genus6-full}
Assume $p=3$ and put
\[
 X_c:\qquad
 y^3-y=c\left(x+\frac1{x^3-x}\right),
 \qquad c\in k^*.
\]
Let
\[
 L=\{(x,y)\mapsto(x,y+r):r\in\F_3\}.
\]
Then $L\triangleleft\Aut(X_c)$ and
\[
 1\longrightarrow L
 \longrightarrow \Aut(X_c)
 \longrightarrow \PGL(2,3)
 \longrightarrow1.
\]
In particular,
\[
 |\Aut(X_c)|=72.
\]
The induced conjugation action of $\PGL(2,3)\cong S_4$ on $L\cong C_3$ is the nontrivial character $S_4\to C_2$. The extension splits, so $\Aut(X_c)\cong C_3\rtimes S_4$ with $S_4$ acting through the sign character. Moreover
\[
 X_c\cong X_{c'}
 \quad\Longleftrightarrow\quad
 c'=\pm c.
\]
\end{proposition}

\begin{proof}
The subgroup $L$ is normal by Lemma~\ref{lem:Lnormal}, since the present curve is the $q=9$ case of the ordinary boundary and
\[
 g=6>(3-1)^2.
\]
The quotient $X_c/L$ is the $x$-line, and the branch set is
\[
 \PP^1(\F_3)=\{\infty,0,1,-1\}.
\]
Hence
\[
 \Aut(X_c)/L
 \le
 \operatorname{Stab}_{\PGL(2,k)}\bigl(\PP^1(\F_3)\bigr)
 =
 \PGL(2,3).
\]

Put
\[
 F(x)=x+\frac1{x^3-x}.
\]
Choose $u\in k$ with $u^3-u=c$. For $t\in\F_3$ the transformations
\[
 (x,y)\longmapsto(x+t,y+tu)
\]
are automorphisms because
\[
 F(x+t)=F(x)+t.
\]
The transformations
\[
 (x,y)\longmapsto(-1/x,y),
 \qquad
 (x,y)\longmapsto(-x,-y)
\]
are automorphisms because
\[
 F(-1/x)=F(x),
 \qquad
 F(-x)=-F(x).
\]
Their images generate $\PGL(2,3)$ on the four branch points. Thus the reduced group is all of $\PGL(2,3)$, and the kernel of the action on the $x$-line is exactly the deck group $L$. Hence $|\Aut(X_c)|=3\cdot24=72$.

The last displayed involution conjugates $y\mapsto y+1$ to $y\mapsto y-1$, so the conjugation action on $L$ is nontrivial. Since $\PGL(2,3)\cong S_4$ has a unique nontrivial linear character, this action is the sign character.

For the splitting, note that $S=\langle\alpha\rangle\times L$ is a Sylow $3$-subgroup of $\Aut(X_c)$, where $\alpha:(x,y)\mapsto(x+1,y+u)$ has order $3$; indeed $|\Aut(X_c)|=8\cdot9$. Thus $L$ has the complement $\langle\alpha\rangle$ in $S$. Since $L$ is an abelian normal $3$-subgroup, Gaschütz's theorem \cite[Hauptsatz~I.17.4]{Hup} shows that $L$ has a complement in $\Aut(X_c)$.

Finally, $L$ is the largest normal $3$-subgroup of $\Aut(X_c)$, because the quotient $S_4$ has no nontrivial normal $3$-subgroup. Hence every isomorphism $X_c\to X_{c'}$ carries $L$ to the corresponding deck group and descends to an element of $\PGL(2,3)$ preserving the branch set. In the Artin--Schreier cohomology of the $x$-line the class of $F$ is fixed by $x\mapsto x+t$ and by $x\mapsto-1/x$, and is negated by $x\mapsto-x$. These transformations generate $\PGL(2,3)$, so the reduced group acts on the line $k[F]$ by the sign character. An isomorphism of the cyclic Artin--Schreier covers may also act on the deck group by multiplication by $\pm1$. Consequently it can carry $c[F]$ to $c'[F]$ exactly when $c'=\pm c$. Conversely $X_c\cong X_{-c}$ via $(x,y)\mapsto(x,-y)$.
\end{proof}

\begin{proof}[Proof of Theorem~\ref{thm:boundary}]
If \eqref{eq:closedbound} is strict, the result is \cite[Theorem~1.1]{GK10}. Assume equality and suppose $X$ is not classical. Proposition~\ref{prop:rank-dichotomy} gives zero $p$-rank or ordinarity. Corollary~\ref{cor:zero-boundary} treats the first case. Proposition~\ref{prop:ordinary-p} treats the ordinary case $|S|=p$, and Proposition~\ref{prop:ordinary-exception} treats $|S|>p$. The converse statements follow from those propositions and Proposition~\ref{prop:cubic-aut}.
\end{proof}

\section{Gaps below the boundary}

Throughout this section put
\[
 q=|S|,\qquad m=\frac{p-1}{p}q,\qquad \gamma=\gamma(X),
\]
and, for a point $R$ in a short $S$-orbit,
\[
 E_P=\sum_{i\ge2}(|S_P^{(i)}|-1),\qquad
 E_R=\sum_{i\ge2}(|(S_R)^{(i)}|-1).
\]

\subsection{A local statement}
The following proposition does not use the hypothesis that $\Aut(X)$ fixes no point. It merges the positive-$p$-rank parts of the proofs of Theorems~\ref{thm:gap} and~\ref{thm:refinedgap}.

\begin{proposition}\label{prop:local}
Let $X$ have genus $g\ge2$, and let $S\le\Aut(X)_P$ be a $p$-subgroup of order $q>p$. Write $g=m+\delta$ with $\delta\ge1$, and assume that either
\begin{enumerate}
\item[(a)] $\delta\le p-2$, or
\item[(b)] $q\ge p^3$ and $\delta<\frac{p(p-1)}2$.
\end{enumerate}
Then $X$ has $p$-rank zero.
\end{proposition}

\begin{proof}
Since $q>p$, one has $q\ge p^2$, and $g<q$: in case (a), $q-g=q/p-\delta\ge2$; in case (b), $q-g\ge p^2-\frac{p(p-1)}2>0$. Let $\bar\gamma$ be the $p$-rank of $X/S$. Besides the fixed point $P$, let the short $S$-orbits have lengths $\ell_1,\ldots,\ell_r$. Deuring--Shafarevich gives
\[
 \gamma=q\bar\gamma+\sum_{i=1}^r(q-\ell_i).
\]
Since $\gamma\le g<q$, one has $\bar\gamma=0$. Every nontrivial stabilizer has order at least $p$, so every summand is at least $q-q/p=m$. As $m\ge p(p-1)>\delta$, there is at most one short orbit besides $\{P\}$.

Assume $\gamma>0$, and denote the additional orbit by $\Delta$. If its stabilizer has order $p^s$ with $s\ge2$, then
\[
 \gamma-m\ge\frac qp-\frac q{p^2}=\frac{q(p-1)}{p^2},
\]
which is at least $p-1>\delta$ in case (a), and at least $p(p-1)>\delta$ in case (b). This contradicts $\gamma\le g=m+\delta$. Hence $s=1$, so
\begin{equation}\label{eq:gap-rank}
 |\Delta|=q/p,\qquad \gamma=m.
\end{equation}

Let $\bar g$ be the genus of $X/S$, and let $R\in\Delta$. Hurwitz gives
\[
 2g-2=q(2\bar g-2)+2(q-1)+E_P
       +\frac qp\bigl(2(p-1)+E_R\bigr).
\]
Using $g=m+\delta$ yields
\begin{equation}\label{eq:gap-excess}
 2\delta=2q\bar g+E_P+\frac qp E_R.
\end{equation}
Since $q>2\delta$, necessarily $\bar g=0$. If $E_R>0$, then $E_R\ge p-1$, and $\frac qp E_R\ge p(p-1)>2\delta$ in both cases, which is impossible. Thus
\begin{equation}\label{eq:EP}
 E_R=0,\qquad E_P=2\delta>0.
\end{equation}

Since $2\delta<p^2-1$ in both cases, the group $Z:=S_P^{(2)}$ has order $p$, and every further nontrivial lower ramification group equals $Z$. Write
\[
 S_P^{(2)}=\cdots=S_P^{(t+1)}=Z,\qquad S_P^{(t+2)}=1,\qquad 2\delta=t(p-1).
\]
In case (a), $t(p-1)\le2p-4$ gives $t=1$. In case (b), $\delta<\frac{p(p-1)}2$ gives $1\le t<p$. The subgroup $Z$ is central in $S$, since it is normal of order $p$ in a $p$-group. Let $L=S_R$, so $|L|=p$.

\emph{Case $Z\ne L$.} At $P$ one has $L\cap S_P^{(2)}=L\cap Z=1$, and at the fixed points of $L$ in $\Delta$ the ramification is weak because $E_R=0$. Hence $X\to X/L$ is weakly ramified at all its ramification points. Let $g_L$ and $\gamma_L$ denote the genus and $p$-rank of $X/L$, and let $r$ be the number of fixed points of $L$. Hurwitz and Deuring--Shafarevich give
\[
 g-1=p(g_L-1)+r(p-1),
 \qquad
 \gamma-1=p(\gamma_L-1)+r(p-1),
\]
and subtracting,
\[
 g-\gamma=p(g_L-\gamma_L).
\]
Thus $p\mid\delta=\frac{t(p-1)}{2}$, so $p\mid t$. This contradicts $1\le t<p$.

\emph{Case $Z=L$.} Since $Z$ is central, it fixes every point of $\Delta$, and its fixed points are exactly $P$ and the points of $\Delta$. Its lower jump at $P$ is $t+1$, and at the points of $\Delta$ it is $1$. Hurwitz for $X\to X/Z$ gives
\[
 2g-2=p\bigl(2g(X/Z)-2\bigr)+(p-1)(t+2)+\frac qp\,2(p-1),
\]
and with $2g=2m+t(p-1)$ this yields $g(X/Z)=0$.

Choose a coordinate $x$ on $X/Z$ so that the image of $P$ is infinity. The group $S/Z$ acts by translations by an additive subgroup $W\subset k$ of order $q/p$, and the branch set is $W\cup\{\infty\}$. Since jumps of cyclic Artin--Schreier extensions of degree $p$ are prime to $p$, we have $t+1\ne p$. Hence $t+1<p$: this is automatic in case (a), where $t=1$, and in case (b) it follows from $t\le p-1$. A reduced equation therefore has the form
\[
 y^p-y=f(x)=\sum_{i=1}^{t+1}a_ix^i+\sum_{w\in W}\frac{c_w}{x-w},\qquad a_{t+1}\ne0.
\]
Because $S/Z$ centralizes the deck group $Z$, translation by every $v\in W$ lifts without changing the generator of $Z$. Comparing simple principal parts shows that all $c_w$ are equal, so the finite polar part is translation-invariant. For $0\ne v\in W$ the difference $f(x+v)-f(x)$ is then a polynomial of degree exactly $t$, with leading coefficient $(t+1)a_{t+1}v\ne0$. It must equal $h(x)^p-h(x)$ for some $h\in k(x)$, which is necessarily a polynomial. Such an expression is constant or has degree divisible by $p$, whereas $1\le t<p$. This contradiction excludes $Z=L$, and proves $\gamma=0$.
\end{proof}

\subsection{A global lemma}
When $\Aut(X)$ fixes no point, the case $Z=L$ above can be excluded for every jump.

\begin{lemma}\label{lem:ZL-global}
Let $X$ have genus $g>(p-1)^2$, put $G=\Aut(X)$, and let $Z\le G$ have order $p$. Assume that:
\begin{itemize}
\item $X/Z\cong\PP^1$;
\item $Z$ has a fixed point $P$ at which its lower ramification jump is at least $2$;
\item its jump at every other fixed point equals $1$.
\end{itemize}
Then $Z\lhd G$, and $G$ fixes $P$.
\end{lemma}

\begin{proof}
If $Z^\sigma\ne Z$ for some $\sigma\in G$, then $k(X)^Z$ and $k(X)^{Z^\sigma}$ are distinct rational subfields of index $p$. Since $k(X)/k(X)^Z$ is cyclic of prime degree, their compositum is $k(X)$. Castelnuovo--Severi then gives $g\le(p-1)^2$, a contradiction. Hence $Z\lhd G$.

Then $G/Z$ acts faithfully on $X/Z\cong\PP^1$ and permutes the branch points of $X\to X/Z$. If $\sigma\in G$ and $Q\in\operatorname{Fix}(Z)$, then $\sigma Q\in\operatorname{Fix}(Z)$, and conjugation by $\sigma$ carries the ramification filtration of $Z$ at $Q$ to that at $\sigma Q$. Hence the jump is preserved.

It follows that $G/Z$ fixes the image of $P$, the unique branch point with jump at least $2$. Since $Z$ fixes $P$, the point $P$ is the only point of $X$ above its image. Therefore $G$ fixes $P$.
\end{proof}

\subsection{Proofs of the gap theorems}

\begin{proof}[Proof of Theorem~\ref{thm:gap}]
Since $q>p$ is a power of $p$, one has $q\ge p^2$. Suppose, towards a contradiction, that
\begin{equation}\label{eq:gap-strip}
 m<g<m+(p-1),
\end{equation}
and write $g=m+\delta$ with $1\le\delta\le p-2$. By Proposition~\ref{prop:local}(a), $X$ has $p$-rank zero. Since $q>g$, Theorem~\ref{thm:zerorank} applies. We show that none of its four families can satisfy the strict strip \eqref{eq:gap-strip}.

For the hyperelliptic family, write its parameter as $N=p^a$. Its Sylow $p$-subgroup has order $N$ and $g=(N-1)/2$. Any proper $p$-subgroup has order at most $N/p\le g$. Hence $q>g$ forces $q=N$, and then
\[
 g<\frac{p-1}{p}q=m,
\]
contrary to \eqref{eq:gap-strip}.

For the Hermitian curve with parameter $N=p^a$, the Sylow order is $N^3$ and $g=N(N-1)/2$. If a $p$-subgroup has order $q>g$, then $q\ge N^2$, because $N^2/p\le g$. Consequently
\[
 m=\frac{p-1}{p}q\ge\frac{p-1}{p}N^2\ge\frac{N(N-1)}2=g,
\]
again impossible.

For the Ree curve, put $N_0=3^r$ and $N=3N_0^2=3^{2r+1}$. Its genus is
\[
 g=\frac32N_0(N-1)(N+N_0+1).
\]

If $r=0$, then $N=3$, $g=15$, and every power of $3$ larger than $g$ is at least $27$, so $m\ge18>g$. Assume $r\ge1$.

Since $N-1>2\cdot3^{2r}$ and $N+N_0+1>3^{2r+1}$,
\[
 g\ge3^{5r+2},
\]
while $N_0+1<N/3$ gives
\[
 g<\frac32N_0\,N\left(N+\frac N3\right)
 =2\cdot3^{5r+2}.
\]
Thus every power of $3$ larger than $g$ is at least $3^{5r+3}$, and for such a subgroup order $q$,
\[
 m=\frac23q\ge2\cdot3^{5r+2}>g.
\]

Finally, for the cubic family $y^3=x^Q-x$ of Theorem~\ref{thm:zerorank}(4), one has $g=Q-1$ and Sylow order $Q$. Since $q>g$ and $q\le Q$, necessarily $q=Q$. Hence
\[
 g-m=\frac qp-1.
\]
As $q>p$ is a power of $p$, $q\ge p^2$, so $g-m\ge p-1$, again contradicting \eqref{eq:gap-strip}.

The strip is therefore empty. Under \eqref{eq:gap-bound}, either $g\le m$, in which case Theorem~\ref{thm:boundary} applies, or \eqref{eq:gap-strip} would hold, which we have just excluded. Since $q>p$, the two boundary families with $q=p$ cannot occur; the only nonclassical boundary case left is the characteristic-$3$ genus-six family.
\end{proof}

In particular, the computation above shows that the hyperelliptic, Hermitian and Ree families satisfy $g\le m$ for every $p$-subgroup of order $q>g$, while the cubic family satisfies $g-m=q/p-1$; we use this below.

\begin{proof}[Proof of Theorem~\ref{thm:refinedgap}]
Assume
\[
 m<g<m+\frac{p(p-1)}2,
 \qquad
 q\ge p^3.
\]
By Proposition~\ref{prop:local}(b), $X$ has $p$-rank zero. Since $q>g$, Theorem~\ref{thm:zerorank} applies. The hyperelliptic, Hermitian and Ree families satisfy $g\le m$. In the cubic family one has
\[
 g-m=\frac qp-1
 \ge p^2-1
 >
 \frac{p(p-1)}2
\]
because $q\ge p^3$. Hence none of the four zero-$p$-rank families lies in the displayed open interval. This proves the theorem.
\end{proof}

\begin{proof}[Proof of Theorem~\ref{thm:beyond}]
Write $g=m+\delta$ with $1\le\delta<\frac{p^2-1}2$. Then $q-g\ge p^2-\delta>0$.

If $\gamma=0$, Theorem~\ref{thm:zerorank} applies. The classical families have $g\le m$, and the cubic family has $g-m=q/p-1\ge p^2-1>\delta$. Both are impossible, so $\gamma>0$.

The first part of the proof of Proposition~\ref{prop:local} applies verbatim, because
\[
 m>\delta,\qquad \frac{q(p-1)}{p^2}\ge p(p-1)>\delta,\qquad \frac qp(p-1)\ge p^2(p-1)>2\delta,\qquad q>2\delta.
\]
It shows that $\bar g=0$, that there is a unique extra short orbit $\Delta$ of length $q/p$, that $\gamma=m$, $E_R=0$ and $E_P=2\delta$. Since $2\delta<p^2-1$, we again have $S_P^{(2)}=\dots=S_P^{(t+1)}=Z$ of order $p$, $S_P^{(t+2)}=1$ and $2\delta=t(p-1)$; now $1\le t\le p$, because $t(p-1)<(p+1)(p-1)$.

If $Z=L$, the Hurwitz computation in the proof of Proposition~\ref{prop:local} shows $X/Z\cong\PP^1$. The jump of $Z$ is $t+1\ge2$ at $P$ and $1$ at the points of $\Delta$, and $g>m\ge p^2(p-1)>(p-1)^2$. Lemma~\ref{lem:ZL-global} then shows that $G$ fixes $P$, contrary to the hypothesis.

Hence $Z\ne L$. As in Proposition~\ref{prop:local}, the cover $X\to X/L$ is weakly ramified and $g-\gamma=p(g_L-\gamma_L)$. Thus $p\mid\delta=\frac{t(p-1)}2$, so $p\mid t$ and $t=p$. Therefore
\[
 \delta=\frac{p(p-1)}2,\qquad g_L-\gamma_L=\frac\delta p=\frac{p-1}2,
\]
and all assertions follow.
\end{proof}

\begin{remark}\label{rem:endpoint}
The argument of Proposition~\ref{prop:local} in the case $Z=L$ breaks down when $t=p$, even locally.

Let $a\in k^*$, let $W=\{v\in k: a^pv^{p^2}+av=0\}$, which is an $\F_p$-plane, let $A_W=\prod_{w\in W}(x-w)$, and let $c\in k^*$. Consider the cover
\[
 y^p-y=a\,x^{p+1}+\frac{c}{A_W(x)}.
\]
\begin{itemize}
\item For $v\in W$ one has $A_W(x+v)=A_W(x)$ and
\[
 a\bigl((x+v)^{p+1}-x^{p+1}\bigr)\equiv\bigl((av)^{1/p}+av^p\bigr)x\equiv0\pmod{\wp(k(x))+k}.
\]
So the translations by $W$ lift, and the curve carries a $p$-group of order $p^3$ fixing the point $P$ above $\infty$.
\item Its genus and $p$-rank are
\[
 g=\frac{p(p-1)(2p+1)}2=m+\frac{p(p-1)}2,\qquad \gamma=m,
\]
and one checks that $S_P^{(2)}=\dots=S_P^{(p+1)}=L$.
\end{itemize}
This configuration has $Z=L$, so Lemma~\ref{lem:ZL-global} shows that its full automorphism group fixes $P$. Whether the configuration $Z\ne L$, $t=p$ of Theorem~\ref{thm:beyond} occurs with $\Aut(X)$ having no global fixed point remains open.
\end{remark}

\begin{proof}[Proof of Theorem~\ref{thm:firstlayer}]
Assume
\[
 g=m+(p-1),\qquad
 m=\frac{p-1}{p}q,\qquad q>p.
\]
If $q\ge p^3$, then $m+(p-1)<m+\frac{p(p-1)}2$, which contradicts Theorem~\ref{thm:refinedgap}. Hence $q=p^2$, and $g<q$.

Let $\gamma$ and $\bar\gamma$ be the $p$-ranks of $X$ and $X/S$. Deuring--Shafarevich gives $\bar\gamma=0$. Every short orbit other than the fixed point contributes at least $m$, while $2m>g$, so there is at most one such orbit. If $\gamma=0$, Theorem~\ref{thm:zerorank} applies because $q>g$. The hyperelliptic, Hermitian and Ree families all satisfy $g\le m$, whereas for the cubic family one has
\[
 g-m=\frac qp-1.
\]
Equality with $p-1$ would require $q=p^2\equiv2\pmod3$, but a square of an odd prime power is never congruent to $2$ modulo $3$. Hence $\gamma>0$.

Let $\Delta$ be the unique additional short orbit and let its stabilizer have order $p^s$, with $s\in\{1,2\}$. Then
\[
 \gamma=q-\frac{q}{p^s}.
\]
If $s=2$, then $\gamma=p^2-1=g$, and $\Delta$ is a second fixed point of $S$.

It remains to exclude $s=1$. Then $|\Delta|=q/p$ and $\gamma=m$. Let $R\in\Delta$. The same Hurwitz calculation as in \eqref{eq:gap-excess} gives
\[
 2(p-1)=2q\bar g+E_P+\frac qpE_R,
\]
where $\bar g=g(X/S)$. Since $q>2(p-1)$, one has $\bar g=0$. If $E_R>0$, then
\[
 \frac qpE_R\ge p(p-1)>2(p-1),
\]
a contradiction. Hence $E_R=0$ and $E_P=2(p-1)$. It follows that
\[
 Z:=S_P^{(2)}=S_P^{(3)}\cong C_p,
 \qquad S_P^{(4)}=1.
\]
The subgroup $Z$ is central in $S$. Put $L=S_R\cong C_p$.

Suppose first that $Z=L$. Then $Z$ fixes every point of $\Delta$. Hurwitz for $X\to X/Z$ gives genus zero. On $X/Z\cong\PP^1$, the group $S/Z$ acts by translations by an additive subgroup $W$ of order $q/p$. The cyclic Artin--Schreier cover $X\to X/Z$ has jump $3$ at infinity and jump $1$ at the finite branch points $W$. If $p=3$, a reduced Artin--Schreier jump cannot be divisible by $p$, so this is impossible. If $p>3$, a reduced equation has the form
\[
 y^p-y=ax^3+bx^2+cx+\sum_{w\in W}\frac{d}{x-w},
 \qquad a\ne0.
\]
The equality of the finite residues follows exactly as in the ordinary boundary argument: translations by $W$ lift and centralize the deck group. Hence the finite polar part is translation-invariant. For $0\ne v\in W$,
\[
 f(x+v)-f(x)=3avx^2+\text{lower terms}=h(x)^p-h(x).
\]
The right-hand side is either constant or has degree divisible by $p$, whereas the left-hand side is a nonconstant quadratic. This is impossible. Alternatively, Lemma~\ref{lem:ZL-global} applies.

Now suppose $Z\ne L$. Every ramification point of the degree-$p$ quotient $X\to X/L$ is weakly ramified: at $P$ one has $L\cap S_P^{(2)}=1$, and at the points of $\Delta$ this follows from $E_R=0$. If $g_L$ and $\gamma_L$ denote the genus and $p$-rank of $X/L$, Hurwitz and Deuring--Shafarevich give
\[
 g-1=p(g_L-1)+r(p-1),\qquad
 \gamma-1=p(\gamma_L-1)+r(p-1),
\]
where $r$ is the number of fixed points of $L$. Thus
\[
 g-\gamma=p(g_L-\gamma_L).
\]
But $g-\gamma=p-1$, a contradiction. Therefore $s=1$ cannot occur.

We have proved
\[
 q=p^2,\qquad g=\gamma=p^2-1.
\]
Let $P_0,P_\infty$ be the two fixed points of $S$. Hurwitz gives
\[
 0=2q\,g(X/S)+E_{P_0}+E_{P_\infty},
\]
so $X/S$ is rational and both excess terms vanish. Thus the cover is totally and weakly ramified at $P_0,P_\infty$ and nowhere else. Since
\[
 S_{P_0}^{(1)}/S_{P_0}^{(2)}=S,
\]
the group $S$ is elementary abelian, hence $S\cong C_p^2$.

We next show that $S$ is a Sylow $p$-subgroup of $G$. Suppose that $S$ is properly contained in a Sylow $p$-subgroup $S^*$. By the normalizer condition for finite $p$-groups,
\[
 S<N_{S^*}(S).
\]
The group $N_{S^*}(S)$ preserves the two-point set $\{P_0,P_\infty\}$, and since $p$ is odd it fixes both points individually. Hence $G_{P_0}$ contains a $p$-subgroup of order at least $p^3$. But
\[
 p^3>\frac{p}{p-1}(p^2-1)=p(p+1),
\]
so the strict theorem of Giulietti--Korchm\'aros applies. None of its three odd-characteristic classical families has genus $p^2-1$, a contradiction. Thus $S$ is Sylow.

Let $N=N_G(S)$. We claim that $N$ contains an element interchanging $P_0$ and $P_\infty$. Suppose not. Then $N$ fixes the two points separately. Since a point stabilizer has a unique Sylow $p$-subgroup, one has
\[
 G_{P_0}=G_{P_\infty}=N.
\]
Put $r=[G:N]$ and write $|N|=p^2h$, with $p\nmid h$. The points $P_0$ and $P_\infty$ lie in two distinct $G$-orbits: if some $g\in G$ sent $P_0$ to $P_\infty$, then $S^g$ and $S$ would both be Sylow $p$-subgroups of $G_{P_\infty}$, hence would coincide, forcing $g\in N$ and contradicting the assumption that $N$ does not interchange the two points. Both orbits have length $r$. Since $X$ is ordinary, the second ramification group at either point is trivial, and the different exponent is
\[
 d=p^2(h+1)-2.
\]
Hurwitz therefore gives
\[
 2(p^2-2)=2g-2
 \ge -2|G|+2rd
 =2r(p^2-2).
\]
Hence $r=1$. Then $N=G$ fixes both $P_0$ and $P_\infty$, contrary to the hypothesis on $G$. This proves the claim.

Choose a coordinate $x$ on $X/S\cong\PP^1$ with the images of $P_0,P_\infty$ equal to $0,\infty$. Elementary Artin--Schreier theory identifies the extension $k(X)/k(x)$ with a two-dimensional $\F_p$-subspace
\[
 V\subset k(x)/\wp(k(x)).
\]
Every nonzero class in $V$ corresponds to a degree-$p$ quotient of the $S$-cover. Since both branch points remain ramified, and the ramification is weak, the unique reduced representative of every nonzero class has a simple pole at both $0$ and $\infty$ and no other pole. Since $k$ is algebraically closed, constants are Artin--Schreier coboundaries. Hence
\[
 V\subset
 \left\{ax+\frac b x:\ a,b\in k\right\},
\]
and every nonzero vector has $a\ne0$ and $b\ne0$.

Let $\tau\in N$ interchange $P_0$ and $P_\infty$. Its induced automorphism of the quotient line is
\[
 x\longmapsto\frac{\lambda}{x}
\]
for some $\lambda\in k^*$. After replacing $x$ by a scalar multiple, we may assume $\lambda=1$. Pullback by $\tau$ preserves $V$ and acts by
\[
 J(a,b)=(b,a).
\]
This proves the invariant description in the statement.

It remains to classify the $J$-invariant planes. Put
\[
 E_+=\{(a,a):a\in k\},
 \qquad
 E_-=\{(a,-a):a\in k\}.
\]
Since $V$ is an $\F_p$-space preserved by the involution $J$ and $\frac12\in\F_p$, every $v\in V$ decomposes as $v=\frac{v+Jv}2+\frac{v-Jv}2$ with both summands in $V$. Thus
\[
 V=(V\cap E_+)\oplus(V\cap E_-).
\]
If one summand has dimension $2$, then, after replacing $x$ by $\iota x$ with $\iota^2=-1$ if necessary, we obtain
\[
 V=A\left(x+\frac1x\right)
\]
for a two-dimensional $\F_p$-subspace $A\subset k$. Otherwise both intersections have dimension $1$, and
\[
 V=
 \F_p\,a\left(x+\frac1x\right)
 \oplus
 \F_p\,b\left(x-\frac1x\right).
\]
The condition that no nonzero vector have a zero coefficient is equivalent to $a/b\notin\F_p$.

Conversely, let $V$ have either of these two forms and put
\[
 k(X)=k(x)\bigl(\wp^{-1}(V)\bigr).
\]
The local coefficient maps at $0$ and $\infty$ are injective on $V$, so the inertia group at each point is all of $C_p^2$, with lower jump $1$. Hurwitz and Deuring--Shafarevich give
\[
 g(X)=\gamma(X)=p^2-1.
\]
In case (1), $x\mapsto1/x$ lifts while fixing chosen Artin--Schreier generators; in case (2), it lifts by fixing the $+$ generator and negating the $-$ generator. Thus the subgroup generated by $S$ and this lift has no common fixed point.

Finally, the curve
\[
 y^{p^2}-y=x+\frac{c}{x}
\]
is a special case of (1): its degree-$p$ subextensions correspond to scalar multiples of the same two-pole function by elements of a two-dimensional $\F_p$-space. Hence the first layer is nonempty, but only for $q=p^2$. The gap theorem is therefore optimal uniformly in $q$; for every fixed $q\ge p^3$ the actual gap is strictly wider.
\end{proof}

\begin{proof}[Proof of Corollary~\ref{cor:firstlayer-moduli}]
First consider marked pairs. An isomorphism conjugating $S$ to $S'$ descends to an isomorphism of the quotient lines and must preserve the unordered branch pair. Hence, after the chosen normalizations, it is of the form
\[
 x\longmapsto\lambda x
 \quad\text{or}\quad
 x\longmapsto\frac{\lambda}{x}.
\]
On reduced Artin--Schreier classes these transformations act, respectively, by
\[
 (a,b)\longmapsto(\lambda^{-1}a,\lambda b)
\]
and by the composition with $J$. Therefore the displayed condition is necessary. Conversely, either transformation carries $V$ onto $V'$, and hence lifts to an isomorphism of the corresponding elementary abelian extensions conjugating $S$ onto $S'$.

Now let $\varphi:X\to X'$ be an arbitrary isomorphism. Since $S$ is a Sylow $p$-subgroup, $\varphi S\varphi^{-1}$ is a Sylow $p$-subgroup of $\Aut(X')$ and is therefore conjugate to $S'$. Postcomposing $\varphi$ with a suitable automorphism of $X'$ reduces to the marked case. Hence the same criterion classifies the underlying curves.

Finally, an automorphism normalizing $S$ descends to a projectivity of $X/S$ preserving the unordered branch pair, and its class in $k^*\rtimes C_2$ must stabilize the Artin--Schreier plane $V$. The kernel of the descent map is exactly the deck group $S$. Conversely, every projectivity in $\operatorname{Stab}_{k^*\rtimes C_2}(V)$ preserves the Artin--Schreier extension and lifts to an automorphism normalizing $S$. This proves the normalizer formula.
\end{proof}

\begin{remark}\label{rem:moduli}
\begin{enumerate}
\item The stabilizer in Corollary~\ref{cor:firstlayer-moduli} always contains $\langle D_{-1},J\rangle\cong C_2\times C_2$, because $-V=V$. Hence $|N_{\Aut(X)}(S)|\ge4p^2$.
\item The two normal forms of Theorem~\ref{thm:firstlayer} are never isomorphic to each other. Indeed, $D_\lambda$ and $J$ are $k$-linear: they carry the $k$-span of $V$, which is a line in case~(1) and all of $k^2$ in case~(2), to the $k$-span of the image.
\item In case~(1), distinct planes $A\ne A'$ give non-isomorphic curves. Indeed, $D_\lambda(E_+)\subseteq E_+$ forces $\lambda^2=1$, and then $D_\lambda V=V$ and $JD_\lambda V=V$ because $-A=A$.
\item In case~(2), suppose $D_\lambda V$ is again $J$-invariant. Since $JD_\lambda=D_{\lambda^{-1}}J$, the plane $V$ is then $D_{\lambda^{-2}}$-invariant. As $V$ is finite and spans $k^2$, the set of such $\lambda$ is finite. Hence each isomorphism class contains only finitely many normalized planes.
\end{enumerate}
Consequently each normal form contributes a two-parameter family of curves: case~(1) is parametrized by the planes $A\subset k$, and case~(2) by the pairs $(a,b)$.
\end{remark}

\section*{Statements and Declarations}
OpenAI's ChatGPT was used for language editing, literature searches,
and assistance with selected calculations and proof steps. The author
takes full responsibility for the mathematical content, references,
and final formulation of the manuscript.

\section*{Acknowledgments}

The author was partially supported by CNPq grant no.~302774/2025-4,
FAPESP grant no.~2024/00923-6, and FAEPEX grant no.~3485/25.

\end{document}